\documentclass[oneside, a4paper, 10pt]{article}
\usepackage[colorlinks = true,
            linkcolor = black,
            urlcolor  = blue,
            citecolor = black]{hyperref}
\usepackage{amsmath, amsthm, amssymb, amsfonts, amscd}
\usepackage[left=3cm,right=3cm,top=3cm,bottom=3cm,a4paper]{geometry}
\usepackage{mathtools}
\usepackage{thmtools}
\usepackage{graphicx}
\usepackage{cleveref}
\usepackage{mathrsfs}
\usepackage{titling}
\usepackage{url}
\usepackage[nosort, noadjust]{cite}
\usepackage{enumitem}
\usepackage{tikz-cd}
\usepackage{multirow}
\usepackage{fancyhdr}
\usepackage{sectsty}
\usepackage{longtable}
\usepackage{authblk}
\DeclareGraphicsExtensions{.pdf,.png,.jpg}

\newcommand{\im}{{\operatorname{im}}}
\newcommand{\cok}{{\operatorname{cok}}}

\newcommand{\Cl}{{\operatorname{Cl}}}
\newcommand{\Gal}{\operatorname{Gal}}
\newcommand{\Hom}{\operatorname{Hom}}

\newcommand{\Aut}{\operatorname{Aut}}
\newcommand{\GL}{\operatorname{GL}}

\newcommand{\Frob}{\operatorname{Frob}}
\newcommand{\rank}{\operatorname{rank}}
\newcommand{\tr}{\operatorname{Tr}}

\newcommand{\diag}{\operatorname{diag}}
\newcommand{\Sur}{\operatorname{Sur}}
\newcommand{\Sym}{\operatorname{Sym}}

\newcommand{\M}{\operatorname{M}}
\newcommand{\HH}{\operatorname{H}}
\newcommand{\Z}{\mathbb{Z}}
\newcommand{\Q}{\mathbb{Q}}

\newcommand{\C}{\mathbb{C}}
\newcommand{\F}{\mathbb{F}}

\newcommand{\PP}{\mathbb{P}}
\newcommand{\EE}{\mathbb{E}}
\newcommand{\OO}{\mathcal{O}}
\newcommand{\YY}{\mathcal{Y}}
\newcommand{\ZZ}{\mathcal{Z}}

\theoremstyle{definition}
\newtheorem{theorem}{Theorem}[section]
\newtheorem{proposition}[theorem]{Proposition}
\newtheorem{corollary}[theorem]{Corollary}
\newtheorem{lemma}[theorem]{Lemma}

\newtheorem*{conjecture*}{Conjecture}
\newtheorem{remark}[theorem]{Remark}
\newtheorem{definition}[theorem]{Definition}

\makeatletter 
\newcommand\semilarge{\@setfontsize\semilarge{11}{13.2}}
\makeatother

\title{\semilarge{\textbf{UNIVERSALITY OF THE COKERNELS OF RANDOM $p$-ADIC HERMITIAN MATRICES}}}
\author{\normalsize{JUNGIN LEE} }
\date{}

\sectionfont{\centering \large}
\subsectionfont{\normalsize}
\makeatletter
\renewcommand{\@seccntformat}[1]{\csname the#1\endcsname.\quad}
\makeatother

\renewenvironment{abstract}
 {\quotation\small\noindent\rule{\linewidth}{.5pt}\par\smallskip
  {\centering\bfseries\abstractname\par}\medskip}
 {\par\noindent\rule{\linewidth}{.5pt}\endquotation}
 
\AtEndDocument{%
  \bigskip
  \footnotesize

  \textsc{Jungin Lee, Department of Mathematics, Ajou University, Suwon 16499, Republic of Korea}\par\nopagebreak
  \textit{E-mail address:} \texttt{jileemath@ajou.ac.kr} }

\begin{document}
\maketitle
\vspace{-18mm}

\begin{abstract}
In this paper, we study the distribution of the cokernel of a general random Hermitian matrix over the ring of integers $\mathcal{O}$ of a quadratic extension $K$ of $\mathbb{Q}_p$. For each positive integer $n$, let $X_n$ be a random $n \times n$ Hermitian matrix over $\mathcal{O}$ whose upper triangular entries are independent and their reductions are not too concentrated on certain values.
We show that the distribution of the cokernel of $X_n$ always converges to the same distribution which does not depend on the choices of $X_n$ as $n \rightarrow \infty$ and provide an explicit formula for the limiting distribution. 
This answers Open Problem 3.16 from the ICM 2022 lecture note of Wood in the case of the ring of integers of a quadratic extension of $\mathbb{Q}_p$. 
\end{abstract}

\section{Introduction} \label{Sec1}

\subsection{Distribution of the cokernel of a random $p$-adic matrix} \label{Sub11}

Let $p$ be a prime. The Cohen-Lenstra heuristics \cite{CL84} predict the distribution of the $p$-Sylow subgroup of the class group $\Cl(K)$ of a random imaginary quadratic field $K$ ordered by the absolute value of the discriminant. (When $p=2$, one needs to modify the conjecture by replacing $\Cl(K)$ with $2 \Cl(K)$.) Friedman and Washington \cite{FW89} computed the limiting distribution of the cokernel of a Haar random $n \times n$ matrix over $\Z_p$ as $n \rightarrow \infty$. They proved that for every finite abelian $p$-group $G$ and a Haar random matrix $A_n \in \M_n(\Z_p)$ for each positive integer $n$, 
\begin{equation*}
\lim_{n \rightarrow \infty} \PP (\cok (A_n) \cong G) = \frac{1}{\left | \Aut(G) \right |} \prod_{i=1}^{\infty}(1-p^{-i})
\end{equation*}
where $\M_n(R)$ denotes the set of $n \times n$ matrices over a commutative ring $R$. The right-hand side of the above formula is equal to the conjectural distribution of the $p$-parts of the class groups of imaginary quadratic fields predicted by Cohen and Lenstra \cite{CL84}.

There are two possible ways to generalize the work of Friedman and Washington. One way is to consider the distribution of the cokernels of various types of random matrices over $\Z_p$. Bhargava, Kane, Lenstra, Poonen and Rains \cite{BKLPR15} computed the distribution of the cokernel of a random alternating matrix over $\Z_p$. 
They suggested a model for the $p$-Sylow subgroup of the Tate-Shafarevich group of a random elliptic curve over $\Q$ of given rank $r \geq 0$, in terms of a random alternating matrix over $\Z_p$. They also proved that the distribution of their random matrix model coincides with the prediction of Delaunay \cite{Del01, Del07, DJ14} on the distribution of the $p$-Sylow subgroup of the Tate-Shafarevich group of a random elliptic curve over $\Q$.
Clancy, Kaplan, Leake, Payne and Wood \cite{CKLPW15} computed the distribution of the cokernel of a random symmetric matrix over $\Z_p$.

In the above results, random matrices are assumed to be equidistributed with respect to Haar measure. The distributions of the cokernels for much larger classes of random matrices were established by Wood \cite{Woo17}, Wood \cite{Woo19} and Nguyen-Wood \cite{NW22}. 

\begin{definition} \label{def1a}
Let $0 < \varepsilon < 1$ be a real number. 
A random variable $x$ in $\Z_p$ is $\varepsilon$\textit{-balanced} if $\PP (x \equiv r \,\, (\text{mod } p)) \leq 1 - \varepsilon$ for every $r \in \Z / p\Z$. 
A random matrix $A$ in $\M_n(\Z_p)$ is $\varepsilon$\textit{-balanced} if its entries are independent and $\varepsilon$-balanced.
A random symmetric matrix $A$ in $\M_n(\Z_p)$ is $\varepsilon$\textit{-balanced} if its upper triangular entries are independent and $\varepsilon$-balanced. 
\end{definition}

\begin{theorem} \label{thm1b}
Let $G$ be a finite abelian $p$-group.
\begin{enumerate}
    \item (\cite[Theorem 1.2]{Woo19}, \cite[Theorem 4.1]{NW22}) Let $(\alpha_n)_{n \geq 1}$ be a sequence of positive real numbers such that $0 < \alpha_n < 1$ for each $n$ and for any constant $\Delta > 0$, we have $\alpha_n \geq \frac{\Delta \log n}{n}$ for sufficiently large $n$. Let $A_n$ be an $\alpha_n$-balanced random matrix in $\M_n(\Z_p)$ for each $n$. Then,
    $$
    \lim_{n \rightarrow \infty} \PP(\cok(A_n) \cong G) = \frac{1}{\left | \Aut(G) \right |} \prod_{i=1}^{\infty}(1-p^{-i}).
    $$
    
    \item (\cite[Theorem 1.3]{Woo17}) Let $0 < \varepsilon < 1$ be a real number and $B_n$ be an $\varepsilon$-balanced random symmetric matrix in $\M_n(\Z_p)$ for each $n$. Then,
    $$
    \lim_{n \rightarrow \infty} \PP(\cok(B_n) \cong G) = \frac{\# \left\{ \text{symmetric, bilinear, perfect } \phi : G \times G \rightarrow \C^* \right\}}{\left | G \right | \left | \Aut(G) \right |} \prod_{i=1}^{\infty}(1-p^{1-2i}).
    $$
\end{enumerate}
\end{theorem}

Another way is to generalize the cokernel condition. We refer to the introduction of \cite{Lee22} for the recent progress in this direction.  The following theorem provides the joint distribution of the cokernels $\cok(P_j(A_n))$ ($1 \leq j \leq l$), where $P_1(t), \cdots, P_l(t) \in \Z_p[t]$ are monic polynomials with some mild assumptions and $A_n$ is a Haar random matrix in $\M_n(\Z_p)$. It is a modified version of the conjecture by Cheong and Huang \cite[Conjecture 2.3]{CH21}.

\begin{theorem} \label{thm1c}
(\cite[Theorem 2.1]{Lee22}) Let $P_1(t), \cdots, P_l(t) \in \Z_p[t]$ be monic polynomials whose mod $p$ reductions in $\F_p[t]$ are distinct and irreducible, and let $G_j$ be a finite module over $R_j := \Z_p[t]/(P_j(t))$ for each $1 \leq j \leq l$. Also let $A_n$ be a Haar random matrix in $\M_n(\Z_p)$ for each positive integer $n$. Then we have
\begin{equation*}
\lim_{n \rightarrow \infty} \PP \begin{pmatrix}
\cok(P_j(A_n)) \cong G_j \\ 
\text{ for } 1 \leq j \leq l
\end{pmatrix} 
= \prod_{j=1}^{l} \left ( \frac{1}{\left | \Aut_{R_j}(G_j) \right |} \prod_{i=1}^{\infty}(1-p^{-i \deg (P_j)}) \right ).
\end{equation*}
\end{theorem}

\subsection{Hermitian matrices over $p$-adic rings} \label{Sub12}

Before stating the main theorem of this paper, we summarize the basic results on quadratic extensions of $\Q_p$ and Hermitian matrices over $p$-adic rings. Every quadratic extension of $\Q_p$ is of the form $\Q_p(\sqrt{a})$ for some non-trivial element $a \in \Q_p^{\times} / (\Q_p^{\times})^2$. For $a, b \in \Q_p^{\times}$, we have $\Q_p(\sqrt{a}) = \Q_p(\sqrt{b})$ if and only if $\frac{b}{a} \in (\Q_p^{\times})^2$. Therefore the number of quadratic extensions of $\Q_p$ is given by $\left | \Q_p^{\times} / (\Q_p^{\times})^2 \right | - 1$. Since $\Q_p^{\times} \cong \Z \times \Z_p \times \Z/(p-1)\Z$ for odd $p$ and $\Q_2^{\times} \cong \Z \times \Z_2 \times \Z/2\Z$, there are $3$ quadratic extensions of $\Q_p$ for odd $p$ and $7$ quadratic extensions of $\Q_2$. 

Let $K$ be a quadratic extension of $\Q_p$ with the ring of integers $\OO := \OO_K$, the residue field $\kappa$ and the uniformizer $\pi$ that will be specified. Denote the generator of the Galois group $\Gal(K/\Q_p)$ by $\sigma$. Fix a primitive $(p^2-1)$-th root of unity $w$ in $\overline{\Q_p}$. Then $K = \Q_p(w)$ is the unique unramified quadratic extension of $\Q_p$ and satisfies $\OO = \Z_p[w]$ and $\kappa = \OO / p \OO \cong \F_{p^2}$ (\cite[Proposition II.7.12]{Neu99}). In this case, fix the uniformizer by $\pi = p$. The element $\sigma \in \Gal(K/\Q_p)$ maps to the Frobenius automorphism $\Frob_p \in \Gal(\F_{p^2}/\F_p)$ ($x \mapsto x^p$) so it satisfies $\sigma(w) = w^p$. 
If $K/\Q_p$ is ramified, then we always have $\OO = \Z_p[\pi]$ and $\kappa = \F_p$. When $K/\Q_p$ is ramified and $p>2$, there exists a uniformizer $\pi \in \OO$ such that $\sigma(\pi) = -\pi$. There are two types of ramified quadratic extensions of $\Q_2$ (\cite[p. 456]{Cho16}):

\begin{enumerate}
    \item $K=\Q_2(\sqrt{1+2u})$ for some $u \in \Z_2^{\times}$, $\pi := 1 + \sqrt{1+2u}$ is a uniformizer and $\sigma (\pi) = 2 - \pi$.
        
    \item $K=\Q_2(\sqrt{2u})$ for some $u \in \Z_2^{\times}$, $\pi := \sqrt{2u}$ is a uniformizer and $\sigma (\pi) = - \pi$.
\end{enumerate}

A matrix $A \in \M_n(\OO)$ is called \textit{Hermitian} if $A = \sigma(A^t)$, where $A^t$ denotes the transpose of $A$. This is equivalent to the condition that $A_{ij} = \sigma(A_{ji})$ for every $1 \leq i \leq j \leq n$, where $A_{ij}$ denotes the $(i,j)$-th entry of the matrix $A$. Denote the set of $n \times n$ Hermitian matrices over $\OO$ by $\HH_n(\OO)$. For an extension of finite fields $\F_{p^2}/\F_p$, the set $\HH_n(\F_{p^2})$ is defined by the same way.

A \textit{Hermitian lattice} over $\OO$ of \textit{rank} $n$ is a free $\OO$-module $L$ of rank $n$ equipped with a bi-additive map $h : L \times L \rightarrow \OO$ such that $h(y, x) = \sigma(h(x,y))$ and $h(ax, by) = a \sigma(b) h(x,y)$ for every $a, b \in \OO$ and $x, y \in L$. 
When $(L, h)$ is a Hermitian lattice with an $\OO$-basis $v_1, \cdots, v_n$, a matrix $H = (H_{i, j})_{1 \leq i,j \leq n} \in \M_n(\OO)$ given by $H_{i, j} = h(v_j, v_i)$ is Hermitian. Conversely, if $H \in \HH_n(\OO)$, then $(L, h)$ given by $L = \OO^n$ and $h(x, y) = \sigma(y^T) H x$ is a Hermitian lattice and we have $h(e_j, e_i) = e_i^T H e_j = H_{i, j}$ where $e_1, \cdots, e_n$ is the standard basis of $L$.

We say two Hermitian matrices $A, B \in \HH_n(\OO)$ are \textit{equivalent} if $B = YA \sigma(Y^t)$ for some $Y \in \GL_n(\OO)$. The correspondence between Hermitian lattices and Hermitian matrices gives a bijection between the set of equivalent classes of Hermitian matrices in $\HH_n(\OO)$ and the set of isomorphism classes of Hermitian lattices over $\OO$ of rank $n$.

\subsection{Main results and the structure of the paper} \label{Sub13}

The purpose of this paper is to establish the universality result for the distribution of the cokernels of random $p$-adic Hermitian matrices. First we provide the definition of $\varepsilon$-balanced random matrix in $\HH_n(\OO)$. 
We will consider the unramified and ramified cases separately. Let $X \in \HH_n(\OO)$. 
If $K/\Q_p$ is unramified, then $X_{ij} = Y_{ij} + w Z_{ij}$ for some $Y_{ij}, Z_{ij} \in \Z_p$. If $K/\Q_p$ is ramified, then $X_{ij} = Y_{ij} + \pi Z_{ij}$ for some $Y_{ij}, Z_{ij} \in \Z_p$. 
For both cases, $Z_{ii}=0$ for each $1 \leq i \leq n$ and $X$ is determined by $n^2$ elements $Y_{ij}$ ($i \leq j$), $Z_{ij}$ ($i < j$). 

Let $E^{ij} \in \M_n(\Z_p)$ be a matrix defined by $(E^{ij})_{kl} = \delta_{ik} \delta_{jl}$.
 If $K/\Q_p$ is unramified, then $\HH_n(\OO)$ is generated by $n^2$ matrices $E^{ij}$ ($ i \le j$) and $wE^{ij}$ ($i < j$) as a $\Z_p$-module. If $K/\Q_p$ is ramified, then $\HH_n(\OO)$ is generated by $n^2$ matrices $E^{ij}$ ($ i \le j$) and $\pi E^{ij}$ ($i < j$) as a $\Z_p$-module. Therefore an additive measure on $\HH_n(\OO)$ defined by the product of the Haar probability measures on $Y_{ij}$ ($i \leq j$), $Z_{ij}$ ($i < j$) is same as the Haar probability measure on $\HH_n(\OO)$ by the uniqueness of the Haar probability measure.

\begin{definition} \label{def1f}
Let $0 < \varepsilon < 1$. A random matrix $X$ in $\HH_n(\OO)$ is \textit{$\varepsilon$-balanced} if the $n^2$ elements $Y_{ij}$ ($i \leq j$), $Z_{ij}$ ($i < j$) in $\Z_p$ are independent and $\varepsilon$-balanced.
\end{definition}

The following remarks shows that our definition of $\varepsilon$-balanced random matrix in $\HH_n(\OO)$ is independent of the choice of the primitive $(p^2-1)$-th root of unity $w$ (unramified case) and the uniformizer $\pi$ (ramified case).

\begin{remark} \label{rmk1g}
\begin{enumerate}
    \item Assume that $K/\Q_p$ is unramified. By the relation
$$
(p+1) + (w^{p-1}-1) \sum_{i=1}^{p+1}i w^{(p+1-i)(p-1)}
= \frac{w^{p^2-1}-1}{w^{p-1}-1} = 0,
$$
we have $w - \sigma(w) = w(1-w^{p-1}) \in \OO^{\times}$. Now let $w'$ be any primitive $(p^2-1)$-th root of unity in $\overline{\Q_p}$. For random elements $x , y \in \Z_p$, we have $x + y w = x' + y' w'$ for 
    $$
    x' = x + \frac{w' \sigma(w) - w \sigma(w')}{w' - \sigma(w')}y, \, y' = \frac{w - \sigma(w)}{w' - \sigma(w')} y.
    $$
    Then we have $x', y' \in \Z_p$ and $x$ and $y$ are independent and $\varepsilon$-balanced if and only if $x'$ and $y'$ are independent and $\varepsilon$-balanced.

    \item Assume that $K/\Q_p$ is ramified and let $\pi'$ be any uniformizer of $K$. Then $\pi' \equiv u \pi \,\, (\text{mod } \pi^2)$ for some $u \in \Z_p^{\times}$. For random elements $x , y \in \Z_p$, we have $x + y \pi' = x' + y' \pi$ for some $x', y' \in \Z_p$ such that $x \equiv x' \,\, (\text{mod } p)$ and $uy \equiv y' \,\, (\text{mod } p)$ so $x$ and $y$ are independent and $\varepsilon$-balanced if and only if $x'$ and $y'$ are independent and $\varepsilon$-balanced. 
\end{enumerate}
\end{remark}

Let $\Gamma$ be an $\OO$-module and ${}^{\sigma} \Gamma$ be its conjugate which is same as $\Gamma$ as abelian groups, with the scalar multiplication $r \cdot g := \sigma(r) g$. 
A \textit{Hermitian pairing} on $\Gamma$ is a bi-additive map $\delta : \Gamma \times \Gamma \rightarrow K/\OO$ such that $\delta(y, x) = \sigma(\delta(x,y))$ and $\delta(ax, by) = a \sigma(b)\delta(x,y)$ for every $a, b \in \OO$ and $x, y \in \Gamma$. 
We say a Hermitian pairing $\delta : \Gamma \times \Gamma \rightarrow K/\OO$ is \textit{perfect} if ${}^{\sigma} \Gamma \rightarrow \Hom_\OO(\Gamma, K/\OO)$ ($g \mapsto \delta(\cdot, g)$) is an $\OO$-module isomorphism. 
The following theorem is the main result of this paper, which settles a problem suggested by Wood \cite[Open Problem 3.16]{Woo22} in the case that $\mathfrak{o}$ is the ring of integers of a quadratic extension of $\Q_p$. 

\begin{theorem} \label{thm1h}
Let $0 < \varepsilon < 1$ be a real number, $X_n \in \HH_n(\OO)$ be an $\varepsilon$-balanced random matrix for each $n$ and $\Gamma$ be a finite $\OO$-module.
\begin{enumerate}
    \item (Theorem \ref{thm4m}) If $K/\Q_p$ is unramified, then
\begin{equation} \label{eq1c}
\lim_{n \rightarrow \infty} \PP(\cok(X_n) \cong \Gamma) 
= \frac{ \# \left\{ \text{Hermitian, perfect } \delta : \Gamma \times \Gamma \rightarrow K/\OO \right\}}{\left| \Aut_{\OO}(\Gamma) \right|} \prod_{i=1}^{\infty}(1 + \frac{(-1)^i}{p^i}).
\end{equation}
    
    \item (Theorem \ref{thm5i}) If $K/\Q_p$ is ramified, then
\begin{equation} \label{eq1d}
\lim_{n \rightarrow \infty} \PP(\cok(X_n) \cong \Gamma) 
= \frac{ \# \left\{ \text{Hermitian, perfect } \delta : \Gamma \times \Gamma \rightarrow K/\OO \right\}}{\left| \Aut_{\OO}(\Gamma) \right|} \prod_{i=1}^{\infty}(1 - \frac{1}{p^{2i-1}}).
\end{equation}
\end{enumerate}
\end{theorem}

In Section \ref{Sec2}, we provide a proof of Theorem \ref{thm1h} under the assumption that each $X_n$ is equidistributed with respect to Haar measure (Theorem \ref{thm2c}). Our proof follows the strategy of \cite[Theorem 2]{CKLPW15} which consists of four steps (see the paragraph after Lemma \ref{lem2f}). The most technical part of the proof is the second step, i.e. the computation of the probability that $\left< \; , \; \right>_{X_n} = \left< \; , \; \right>_M$ for a given $M \in \HH_n(\OO)$. When $K/\Q_p$ is ramified, this computation is even more complicated than the proof of \cite[Theorem 2]{CKLPW15} for $p=2$.

To extend this result to $\varepsilon$-balanced Hermitian matrices, we use the \textit{moments} as in the symmetric case. For a random $\OO$-module $M$ and a given $\OO$-module $G$, the $G$\textit{-moment} of $M$ is defined by the expected value $\mathbb{E}(\# \Sur_{\OO}(M, G))$ of the number of surjective $\OO$-module homomorphisms from $M$ to $G$. 
The key point is that if we know the $G$-moment of $M$ for every $G$ and if the moments are not too large, then we can recover the distribution of a random $\OO$-module $M$. In Section \ref{Sec3}, we show that the limiting distribution of the cokernels of Haar random Hermitian matrices over $\OO$ is determined by their moments (Theorem \ref{thm3c}). We conclude the proof of Theorem \ref{thm1h} in the general case by combining the following result with Theorem \ref{thm2c} and \ref{thm3c}.

\begin{theorem} \label{thm1i}
Let $X_n$ be as in Theorem \ref{thm1h} and $G = \prod_{i=1}^{r} \OO / \pi^{\lambda_i} \OO$ for $\lambda_1 \geq \cdots \geq \lambda_r \geq 1$. 
\begin{enumerate}
    \item (Theorem \ref{thm4l}) If $K/\Q_p$ is unramified, then
    \begin{equation} \label{eq1e}
    \lim_{n \rightarrow \infty} \EE(\# \Sur_{\OO}(\cok(X_n), G)) = p^{\sum_{i=1}^{r} (2i-1)\lambda_i}.
    \end{equation}
    
    \item (Theorem \ref{thm5h}) If $K/\Q_p$ is ramified, then 
    \begin{equation} \label{eq1f}
    \lim_{n \rightarrow \infty} \EE(\# \Sur_{\OO}(\cok(X_n), G)) = p^{\sum_{i=1}^{r}\left ( (i-1)\lambda_i + \left \lfloor \frac{\lambda_i}{2}\right \rfloor \right )}.
    \end{equation}
\end{enumerate}
For both cases, the error term is exponentially small in $n$.
\end{theorem}

The unramified and ramified cases should be considered separately in the proof of the above theorem. Moreover, the ramified extensions $K/\Q_p$ are classified by two types (see the first paragraph of Section \ref{Sec5}) and the proof for these cases are slightly different for some technical reasons. 
We prove the unramified case in Section \ref{Sec4} and the ramified case in Section \ref{Sec5}. Our proof of Theorem \ref{thm1i} is based on the innovative work of Wood \cite{Woo17}, but it cannot be directly adapted to our case.
In particular, we need some effort to deal with the linear and conjugate-linear maps simultaneously, which is different from the symmetric case where every map is linear. For example, a lot of conjugations appear during the computations on the maps $\alpha_{ij}^{c}$ and $\alpha_i^{d}$ in $\Hom_R(V, ({}^{\sigma}G)^*)$, which make the proof more involved than the proof for the symmetric case.

\section{Haar random $p$-adic Hermitian matrices} \label{Sec2}

Throughout this section, we assume that random matrices are equidistributed with respect to Haar measure. More general (i.e. $\varepsilon$-balanced) random matrices will be considered in Section \ref{Sec4} and \ref{Sec5}. This section is based on \cite[Section 2]{CKLPW15}, where the distribution of the cokernel of a Haar random symmetric matrix over $\Z_p$ was computed. Some notations are also borrowed from \cite{CKLPW15}.

Jacobowitz \cite{Jac62} classified Hermitian lattices over $p$-adic rings. We follow the exposition of Yu \cite[Section 2]{Yu12}, whose original reference is also \cite{Jac62}. Let $(L, h)$ be a Hermitian lattice over $\OO$. (We refer Section \ref{Sub12} for the definition of a Hermitian lattice.) A vector $x \in L$ is \textit{maximal} if $x \notin \pi L$. (Recall that $\pi$ is a uniformizer of $\OO$.) For each $i \in \Z$, $(L, h)$ is called $\pi^i$-\textit{modular} if $h(x, L) = \pi^i \OO$ for every maximal $x \in L$ (\cite[p.447]{Jac62}). We say $(L, h)$ is \textit{modular} if it is $\pi^i$-modular for some $i \in \Z$. Any Hermitian lattice can be written as an orthogonal sum of modular lattices \cite[Proposition 4.3]{Jac62}. When $K/\Q_p$ is unramified, any $\pi^i$-modular lattice $(L, h)$ is isomorphic to an orthogonal sum of the copies of the $\pi^i$-modular lattice $(\pi^i)$ of rank $1$ \cite[Theorem 7.1]{Jac62}. Using the correspondence between Hermitian lattices and Hermitian matrices, we obtain the following proposition.

\begin{proposition} \label{prop2a}
Assume that $K/\Q_p$ is unramified. For every $A \in \HH_n(\OO)$, there exists $Y \in \GL_n(\OO)$ such that $YA \sigma(Y^t)$ is diagonal. 
\end{proposition}

The classification of Hermitian lattices over $\OO$ for the ramified case can be found in \cite[Proposition 8.1]{Jac62} (the case $p>2$) and \cite[Theorem 2.2]{Cho16} (the case $p=2$). Using the correspondence between Hermitian lattices and Hermitian matrices, we can summarize the results of \cite[Proposition 8.1]{Jac62} and \cite[Theorem 2.2]{Cho16} as follow. For $a, b \in \Z_p$ and $c \in \OO$, denote
$$
A(a,b,c) := \begin{pmatrix}
a & c \\
\sigma (c) & b \\
\end{pmatrix} \in \HH_2(\OO).
$$

\begin{proposition} \label{prop2b}
Assume that $K/\Q_p$ is ramified. For every $A \in \HH_n(\OO)$, there exists $Y \in \GL_n(\OO)$ such that $YA \sigma(Y^t)$ is a block diagonal matrix consisting of
\begin{enumerate}
    \item the zero diagonal blocks,
    
    \item diagonal blocks $\begin{pmatrix} u_ip^{d_i} \end{pmatrix}$ for some $d_i \geq 0$ and $u_i \in \Z_p^{\times}$,
    
    \item $2 \times 2$ blocks of the form $B_j = A(a_j, b_j, c_j)$ for some $a_j, b_j \in \Z_p$ and $c_j \in \OO$ such that $c_j \neq 0$, $\displaystyle \frac{a_j}{c_j} \in \OO$ and $\displaystyle \frac{b_j}{\pi c_j} \in \OO$.
\end{enumerate}
\end{proposition}

Let $\Gamma$ be a finite $\OO$-module. 
Recall that we have defined the conjugate ${}^{\sigma} \Gamma$ of $\Gamma$ and a perfect Hermitian pairing on $\Gamma$ in Section \ref{Sub13}. 
When $\Gamma_1$, $\Gamma_2$ are finite $\OO$-modules and $\delta_i$ ($i = 1, 2$) is a perfect Hermitian pairing on $\Gamma_i$, we say $(\Gamma_1, \delta_1)$ and $(\Gamma_2, \delta_2)$ are \textit{isomorphic}  if there is an $\OO$-module isomorphism $f : \Gamma_1 \rightarrow \Gamma_2$ such that $\delta_1 (x, y) = \delta_2 ( f(x), f(y))$ for every $x, y \in \Gamma_1$. 
For $X \in \HH_n(\OO)$, a Hermitian pairing $\left< \; , \; \right>_X : \OO^n \times \OO^n \rightarrow K/\OO$ is defined by
$$
\left< x, \, y \right>_X := \sigma(y^t) X^{-1} x
$$
and this induces a perfect Hermitian pairing $\delta_X$ on $\cok(X)$. 
The following theorem provides the limiting distribution of the cokernel of a Haar random Hermitian matrix over $\OO$. 

\begin{theorem} \label{thm2c}
Let $X_n \in \HH_n(\OO)$ be a Haar random matrix for each $n$, $\Gamma$ be a finite $\OO$-module, $r := \dim_{\kappa}(\Gamma / \pi \Gamma)$ and $\delta$ be a perfect Hermitian pairing on $\Gamma$.
\begin{enumerate}
    \item If $K/\Q_p$ is unramified, then the probability that $(\cok(X_n), \delta_{X_n})$ is isomorphic to $(\Gamma, \delta)$ is 
\begin{equation*}
\mu_n (\Gamma, \delta) 
= \frac{1}{\left| \Aut_{\OO}(\Gamma, \delta) \right|} \prod_{j=n-r+1}^{n}(1-\frac{1}{p^{2j}}) \prod_{i=1}^{n-r}(1 + \frac{(-1)^i}{p^i}),
\end{equation*}
which implies that
\begin{equation} \label{eq2c}
\lim_{n \rightarrow \infty} \PP(\cok(X_n) \cong \Gamma) 
= \frac{ \# \left\{ \text{Hermitian, perfect } \delta : \Gamma \times \Gamma \rightarrow K/\OO \right\}}{\left| \Aut_{\OO}(\Gamma) \right|} \prod_{i=1}^{\infty}(1 + \frac{(-1)^i}{p^i}).
\end{equation}
    
    \item If $K/\Q_p$ is ramified, then the probability that $(\cok(X_n), \delta_{X_n})$ is isomorphic to $(\Gamma, \delta)$ is 
\begin{equation*}
\mu_n (\Gamma, \delta) 
= \frac{1}{\left| \Aut_{\OO}(\Gamma, \delta) \right|} \prod_{j=n-r+1}^{n}(1-\frac{1}{p^j}) \prod_{i=1}^{\left \lceil \frac{n-r}{2} \right \rceil}(1 - \frac{1}{p^{2i-1}}),
\end{equation*}
which implies that
\begin{equation} \label{eq2d}
\lim_{n \rightarrow \infty} \PP(\cok(X_n) \cong \Gamma) 
= \frac{ \# \left\{ \text{Hermitian, perfect } \delta : \Gamma \times \Gamma \rightarrow K/\OO \right\}}{\left| \Aut_{\OO}(\Gamma) \right|} \prod_{i=1}^{\infty}(1 - \frac{1}{p^{2i-1}}).
\end{equation}
\end{enumerate}
\end{theorem}

Before starting the proof, we provide three lemmas that will be used in the proof. The first one is an analogue of \cite[Lemma 4]{CKLPW15}, whose proof is also similar. For an arbitrary Hermitian pairing $\left [ \; , \; \right ] : \OO^n \times \OO^n \rightarrow K/\OO$, define its cokernel by
$$
\cok \left [ \; , \; \right ] := \OO^n / \left\{ x \in \OO^n : \left [ x, y \right ]=0 \text{ for all } y \in \OO^n \right\}
$$
and denote the canonical perfect Hermitian pairing on $\cok \left [ \; , \; \right ]$ by $\delta_{\left [ \; , \; \right ]}$. Note that $\cok \left [ \; , \; \right ]$ is always a finite $\OO$-module.

\begin{lemma} \label{lem2d}
The number of Hermitian pairings $\left [ \; , \; \right ] : \OO^n \times \OO^n \rightarrow K/\OO$ such that $(\cok \left [ \; , \; \right ], \delta_{\left [ \; , \; \right ]})$ is isomorphic to given $(\Gamma, \delta)$ is
$$
\frac{\left| \Gamma \right|^n \prod_{j=n-r+1}^{n} (1 - \left| \kappa \right|^{-j})}{\left| \Aut_{\OO}(\Gamma, \delta) \right|},
$$
where $r := \dim_{\kappa}(\Gamma / \pi \Gamma)$.
\end{lemma}

Denote the set of $n \times n$ symmetric matrices over a commutative ring $R$ by $\Sym_n(R)$. For a matrix $A \in \HH_n(\OO)$, $\overline{A} := A \text{ mod } (\pi)$ is an element of $\HH_n(\F_{p^2})$ (resp. $\Sym_n(\F_{p})$) if $K/\Q_p$ is unramified (resp. ramified).

\begin{lemma} \label{lem2e}
\begin{enumerate} 
    \item (\cite[equation (4)]{GLSV14}) The number of invertible matrices in $\HH_n(\F_{p^2})$ is
\begin{equation} \label{eq2b}
\left| \GL_n(\F_{p^2}) \cap \HH_n(\F_{p^2}) \right| = 
p^{n^2} \prod_{i=1}^{n} (1 + \frac{(-1)^i}{p^i}).
\end{equation}

\item (\cite[Theorem 2]{Mac69}) The number of invertible matrices in $\Sym_n(\F_p)$ is
\begin{equation} \label{eq2a}
\left| \GL_n(\F_p) \cap \Sym_n(\F_p) \right| = 
p^{\frac{n(n+1)}{2}} \prod_{i=1}^{\left \lceil \frac{n}{2} \right \rceil} (1 - \frac{1}{p^{2i-1}}).
\end{equation}
\end{enumerate}
\end{lemma}

The following lemma is a variant of \cite[Lemma 3.2]{BKLPR15}. Since an $n \times n$ matrix over $\OO$ is invertible if and only if its reduction modulo $\pi$ (which is a matrix over $\kappa$) is invertible, the proof is exactly same as the original lemma.

\begin{lemma} \label{lem2f}
Suppose that $A, M \in \HH_n(\OO)$ and $\det M \neq 0$. Then we have $\left< \; , \; \right>_A = \left< \; , \; \right>_M$ if and only if $A \in M + M \HH_n(\OO) M$ and $\rank_{\kappa}(\overline{A}) = \rank_{\kappa}(\overline{M})$.
\end{lemma}

Now we are ready to prove Theorem \ref{thm2c}, which is the main result of this section. Our proof follows the strategy of \cite[Theorem 2]{CKLPW15}. 

\begin{proof}[Proof of Theorem \ref{thm2c}]
Throughout the proof, we denote a Haar random matrix $X_n \in \HH_n(\OO)$ by $A$ for simplicity. First we prove the case that $K/\Q_p$ is unramified. 

\begin{enumerate}
    \item By Lemma \ref{lem2d}, the number of Hermitian pairings $\left [ \; , \; \right ] : \OO^n \times \OO^n \rightarrow K/\OO$ such that $(\cok \left [ \; , \; \right ], \delta_{\left [ \; , \; \right ]})$ is isomorphic to $(\Gamma, \delta)$ is
    $$
    \frac{\left| \Gamma \right|^n}{\left| \Aut_{\OO}(\Gamma, \delta) \right|} \prod_{j=n-r+1}^{n} (1 - \frac{1}{p^{2j}}).
    $$

    \item Let $\left [ \; , \; \right ] : \OO^n \times \OO^n \rightarrow K/\OO$ be a Hermitian pairing. Choose a matrix $N \in \HH_n(K)$ such that $N_{ij} \in K$ is a lift of $\left [ e_i, e_j \right ] \in K/\OO$. There exists $m \in \Z_{\geq 0}$ such that $p^m N \in \HH_n(\OO)$. By Proposition \ref{prop2a}, there exists $Y \in \GL_n(\OO)$ such that
    $$
    YN\sigma(Y^t) = \diag(u_1'p^{d_1'-m}, \cdots, u_n'p^{d_n'-m}),
    $$
    where $u_i' \in \OO^{\times}$ and $d_i' \in \Z_{\geq 0}$ for each $i$.
    By possibly changing the lift $N_{ij}$ of $\left [ e_i, e_j \right ] \in K/\OO$ for each $1 \leq i \leq j \leq n$, one may assume that $d_i' - m \leq 0$ for each $i$. In this case, $M := (YN\sigma(Y^t))^{-1} \in \HH_n(\OO)$ satisfies $\left [ \; , \; \right ] = \left< \; , \; \right>_M$ and is of the form
    $$
M = \diag (u_1 p^{d_1}, \cdots, u_n p^{d_n}),
$$
    where $u_i \in \OO^{\times}$ and $d_i \in \Z_{\geq 0}$ for each $i$.

    \item Let $M \in \HH_n(\OO)$ be as above. Then we have $\Gamma \cong \cok(M) \cong \prod_{i=1}^{n} \OO / p^{d_i} \OO$. By Lemma \ref{lem2f}, we have $\left< \; , \; \right>_A = \left< \; , \; \right>_M$ if and only if $X := M^{-1}(A-M)M^{-1} \in \HH_n(\OO)$ and $\rank_{\F_{p^2}}(\overline{A}) = \rank_{\F_{p^2}}(\overline{M}) = n-r$. First we compute the probability that $X_{ij} \in \OO$ for every $i \leq j$. 

\begin{itemize}
    \item For $1 \leq i < j \leq n$, we have $X_{ij} = (u_i p^{d_i})^{-1} A_{ij} (u_{j} p^{d_{j}})^{-1}\in \OO$ if and only if $p^{d_i+d_j} \mid A_{ij}$ for a random $A_{ij} \in \OO$. The probability is given by $p^{-2(d_i+d_j)}$.
    
    \item For $1 \leq i \leq n$, we have $X_{ii} = (u_i p^{d_i})^{-1} A_{ii} (u_{i} p^{d_{i}})^{-1}\in \OO$ if and only if $p^{2d_i} \mid A_{ii}$ for a random $A_{ii} \in \Z_p$. The probability is given by $p^{-2d_i}$.
\end{itemize}
By the above computations, the probability that $X \in \HH_n(\OO)$ is given by 
\begin{equation} \label{eq2e1}
\prod_{i<j}p^{-2(d_i+d_j)} \prod_{i} p^{-2d_i} =  \left| \Gamma \right|^{-n}.
\end{equation}
Given the condition $X \in \HH_n(\OO)$, we may assume that $\overline{M}$ is zero outside of its upper left $(n-r) \times (n-r)$ minor by permuting the rows and columns of $M$. In this case, $\rank_{\F_{p^2}}(\overline{A})= n-r$ if and only if the upper left $(n-r) \times (n-r)$ minor of $\overline{A}$, which is random in $\HH_{n-r}(\F_{p^2})$, has a rank $n-r$. Therefore the probability that $\left< \; , \; \right>_A = \left< \; , \; \right>_M$ is given by
\begin{equation*}
\frac{\left| \GL_{n-r}(\F_{p^2}) \cap \HH_{n-r}(\F_{p^2}) \right|}{\left| \HH_{n-r}(\F_{p^2}) \right|} \cdot \left| \Gamma \right|^{-n}
= \prod_{i=1}^{n-r} (1 + \frac{(-1)^i}{p^i})\left| \Gamma \right|^{-n}.
\end{equation*}
(The equality holds due to Lemma \ref{lem2e}.)

    \item Using the above results, we conclude that
    \begin{equation*}
\begin{split}
\mu_n(\Gamma, \delta) 
& = \prod_{i=1}^{n-r} (1 + \frac{(-1)^i}{p^i})\left| \Gamma \right|^{-n} \cdot \frac{\left| \Gamma \right|^n}{\left| \Aut_{\OO}(\Gamma, \delta) \right|} \prod_{j=n-r+1}^{n} (1 - \frac{1}{p^{2j}}) \\
& = \frac{1}{\left| \Aut_{\OO}(\Gamma, \delta) \right|} \prod_{j=n-r+1}^{n}(1-\frac{1}{p^{2j}}) \prod_{i=1}^{n-r} (1 + \frac{(-1)^i}{p^i}).
\end{split}    
\end{equation*}
Let $\Phi_{\Gamma}$ denotes the set of Hermitian perfect pairings on $\Gamma$ and $\overline{\Phi}_{\Gamma}$ be the set of isomorphism classes of elements of $\Phi_{\Gamma}$. The orbit-stabilizer theorem implies that
    \begin{equation} \label{eq2x1}
    \sum_{[\delta] \in \overline{\Phi}_{\Gamma}} \frac{1}{\left| \Aut_{\OO}(\Gamma, \delta) \right|}
    = \sum_{\delta \in \Phi_{\Gamma}} \frac{1}{\left| \Aut_{\OO}(\Gamma) \right|}
    \end{equation}
    (\cite[p. 952]{Woo17})
and this implies that
\begin{equation*}
\lim_{n \rightarrow \infty} \PP(\cok(A) \cong \Gamma) 
= \lim_{n \rightarrow \infty} \sum_{[\delta] \in \overline{\Phi}_{\Gamma}} \mu_n(\Gamma, \delta)
= \frac{\left| \Phi_{\Gamma} \right|}{\left| \Aut_{\OO}(\Gamma) \right|} \prod_{i=1}^{\infty}(1 + \frac{(-1)^i}{p^i}). 
\end{equation*}
\end{enumerate}

\noindent Now assume that $K/\Q_p$ is ramified. 
\begin{enumerate}
    \item By Lemma \ref{lem2d}, the number of Hermitian pairings $\left [ \; , \; \right ] : \OO^n \times \OO^n \rightarrow K/\OO$ such that $(\cok \left [ \; , \; \right ], \delta_{\left [ \; , \; \right ]})$ is isomorphic to $(\Gamma, \delta)$ is
$$
\frac{\left| \Gamma \right|^n}{\left| \Aut_{\OO}(\Gamma, \delta) \right|} \prod_{j=n-r+1}^{n} (1 - \frac{1}{p^j}).
$$

    \item Let $\left [ \; , \; \right ] : \OO^n \times \OO^n \rightarrow K/\OO$ be a Hermitian pairing. Choose a matrix $N \in \HH_n(K)$ such that $N_{ij} \in K$ is a lift of $\left [ e_i, e_j \right ] \in K/\OO$. There exists $m \in \Z_{\geq 0}$ such that $p^m N \in \HH_n(\OO)$. By Proposition \ref{prop2b}, there exists $Y \in \GL_n(\OO)$ such that
    $$
    YN\sigma(Y^t) = \diag (u_1' p^{d_1'-m}, \cdots, u_k' p^{d_k'-m}, p^{-m}B_1', \cdots, p^{-m}B_s') \;\; (k+2s=n),
    $$
    where $u_i'p^{d_i'}$ ($1 \leq i \leq k$) and $B_j'=A(a_j', b_j', c_j')$ ($1 \leq j \leq s$) satisfy the conditions appear in Proposition \ref{prop2b}. 
    By possibly changing the lift $N_{ij}$ of $\left [ e_i, e_j \right ] \in K/\OO$ for each $1 \leq i \leq j \leq n$, one may assume that $d_i' - m \leq 0$ for each $i$ and $(p^{-m}B_j')^{-1} =p^m A(b_j', a_j', -c_j') \in \HH_2(\OO)$ for each $j$. In this case, $M := (YN\sigma(Y^t))^{-1} \in \HH_n(\OO)$ satisfies $\left [ \; , \; \right ] = \left< \; , \; \right>_M$ and is of the form
    $$
M = \diag (u_1 p^{d_1}, \cdots, u_k p^{d_k}, B_1, \cdots, B_s) \;\; (k+2s=n),
$$
    where $u_ip^{d_i} \in \Z_p$ ($1 \leq i \leq k$) and $B_j = A(a_j, b_j, c_j) \in \HH_2(\OO)$ ($1 \leq j \leq s$) satisfy the conditions appear in Proposition \ref{prop2b}.

    \item Let $M \in \HH_n(\OO)$ be as above. The conditions $c_j \neq 0$, $\displaystyle \frac{a_j}{c_j} \in \OO$ and $\displaystyle \frac{b_j}{\pi c_j} \in \OO$ imply that
$$
\cok (B_j) \cong \cok \begin{pmatrix}
a_j & c_j \\
\sigma(c_j)-a_jb_jc_j^{-1} & 0 \\
\end{pmatrix}
\cong \cok \begin{pmatrix}
0 & c_j \\
\sigma(c_j)-a_jb_jc_j^{-1} & 0 \\
\end{pmatrix}
\cong (\OO/c_j\OO)^2
$$
so we have
$$
\Gamma \cong \cok(M) \cong \prod_{i=1}^{k} \OO / \pi^{2d_i} \OO \times \prod_{j=1}^{s} (\OO / \pi^{e_j} \OO)^2
$$
for $\pi^{e_j} \parallel c_j$. By Lemma \ref{lem2f}, we have $\left< \; , \; \right>_A = \left< \; , \; \right>_M$ if and only if $X := M^{-1}(A-M)M^{-1} \in \HH_n(\OO)$ and $\rank_{\F_p}(\overline{A}) = \rank_{\F_p}(\overline{M}) = n-r$. First we compute the probability that $X_{ij} \in \OO$ for every $i \leq j$. To simplify the proof, we will introduce some notations. For a positive integer $x$, denote $\underline{x} := x+k$. For $1 \leq i \leq k$ and $1 \leq j \leq s$, denote $\YY_{ij} := \begin{pmatrix}
X_{i, \, \underline{2j-1}} & X_{i, \, \underline{2j}} \\
\end{pmatrix}$. For $1 \leq j \leq j' \leq s$, denote $\ZZ_{j j'} := \begin{pmatrix}
X_{\underline{2j-1}, \, \underline{2j'-1}} & X_{\underline{2j-1}, \, \underline{2j'}} \\
X_{\underline{2j}, \, \underline{2j'-1}} & X_{\underline{2j}, \, \underline{2j'}} \\
\end{pmatrix}$.

\begin{itemize}
    \item For $1 \leq i < i' \leq k$, we have $X_{ii'} = (u_i p^{d_i})^{-1} A_{ii'} (u_{i'} p^{d_{i'}})^{-1}\in \OO$ if and only if $\pi^{2(d_i+d_{i'})} \mid A_{ii'}$ for a random $A_{ii'} \in \OO$. The probability is given by $p^{-2(d_i+d_{i'})}$.
    
    \item For $1 \leq i \leq k$, we have $X_{ii} = (u_i p^{d_i})^{-1}(A_{ii} - u_i p^{d_i})(u_i p^{d_i})^{-1} \in \OO$ if and only if $p^{2d_i} \mid A_{ii} - u_ip^{d_i}$ for a random $A_{ii} \in \Z_p$. The probability is given by $p^{-2d_i}$.
    
    \item For $1 \leq i \leq k$ and $1 \leq j \leq s$, denote $(a_j, b_j, c_j)=(a,b,c)$ for simplicity. Then we have 
\begin{equation*}
\begin{split}
& \YY_{ij} = (u_i p^{d_i})^{-1} \begin{pmatrix}
A_{i, \, \underline{2j-1}} = x & A_{i, \, \underline{2j}} = y \\
\end{pmatrix} B_j^{-1} \in \M_{1 \times 2}(\OO) \\
\Leftrightarrow \; & -\frac{b}{\sigma(c)} x + y, \, x + -\frac{a}{c} y \in \pi^{2d_i+e_j} \OO \\
\Leftrightarrow \; & x, y \in \pi^{2d_i+e_j} \OO.
\end{split}    
\end{equation*}
The probability is given by $p^{-2(2d_i+e_j)}$.
    
    \item For $1 \leq j < j' \leq s$, denote $(a_j, b_j, c_j)=(a,b,c)$ and $(a_{j'}, b_{j'}, c_{j'})=(a',b',c')$ for simplicity. Then we have 
\begin{equation*}
\begin{split}
& \ZZ_{jj'} = B_j^{-1} \begin{pmatrix}
A_{\underline{2j-1}, \, \underline{2j'-1}} = x & A_{\underline{2j-1}, \, \underline{2j'}} = y \\
A_{\underline{2j}, \, \underline{2j'-1}} = z & A_{\underline{2j}, \, \underline{2j'}} = w \\
\end{pmatrix} B_{j'}^{-1} \in \M_{2}(\OO) \\
\Leftrightarrow \; & \begin{pmatrix}
bb' & -b \sigma(c') & -cb' & c \sigma(c') \\
-bc' & ba' & cc' & -ca' \\
- \sigma(c) b' & \sigma(c) \sigma(c') & ab' & -a \sigma(c') \\
\sigma(c) c' & -\sigma(c) a' & -ac' & aa' \\
\end{pmatrix} \begin{pmatrix}
x \\
y \\
z \\
w\end{pmatrix} \in \pi^{2(e_j+e_{j'})} \M_{4 \times 1}(\OO) \\
\Leftrightarrow \;\, & T \begin{pmatrix}
x \\
y \\
z \\
w\end{pmatrix} := \begin{pmatrix}
\frac{bb'}{c \sigma(c')} & -\frac{b}{c} & -\frac{b'}{\sigma(c')} & 1 \\
-\frac{b}{c} & \frac{ba'}{cc'} & 1 & -\frac{a'}{c'}  \\
-\frac{b'}{\sigma(c')} & 1 & \frac{ab'}{\sigma(c) \sigma(c')} & -\frac{a}{\sigma(c)} \\
1 & -\frac{a'}{c'} & -\frac{a}{\sigma(c)} & \frac{aa'}{\sigma(c) c'} \\
\end{pmatrix} \begin{pmatrix}
x \\
y \\
z \\
w\end{pmatrix} \in \pi^{e_j+e_{j'}} \M_{4 \times 1}(\OO).
\end{split}
\end{equation*}
Since the matrix $T \in \M_4(\OO)$ is invertible, the probability is given by $p^{-4(e_j+e_{j'})}$.
    
    \item For $1 \leq j \leq s$, denote $(a_j, b_j, c_j)=(a,b,c)$ for simplicity. Then we have 
    \begin{equation*}
\begin{split}
& \ZZ_{jj} = B_j^{-1} \begin{pmatrix}
A_{\underline{2j-1}, \, \underline{2j-1}} = x & A_{\underline{2j-1}, \, \underline{2j'}} = y \\
A_{\underline{2j}, \, \underline{2j-1}} = \sigma(y) & A_{\underline{2j}, \, \underline{2j'}} = z \\
\end{pmatrix} B_{j}^{-1} \in \M_{2}(\OO) \\
\Leftrightarrow \; & \begin{pmatrix}
\frac{b^2}{c \sigma(c)} & -\frac{b}{c} & -\frac{b}{\sigma(c)} & 1 \\
-\frac{b}{c} & \frac{ba}{c^2} & 1 & -\frac{a}{c}  \\
-\frac{b}{\sigma(c)} & 1 & \frac{ab}{\sigma(c)^2} & -\frac{a}{\sigma(c)} \\
1 & -\frac{a}{c} & -\frac{a}{\sigma(c)} & \frac{a^2}{\sigma(c) c} \\
\end{pmatrix} \begin{pmatrix}
x \\
y \\
\sigma(y) \\
z \end{pmatrix} \in \pi^{2e_j} \M_{4 \times 1}(\OO) \\
\Leftrightarrow \;\, & x, z \in p^{e_j} \Z_p, \, y \in \pi^{2e_j} \OO
\end{split}
\end{equation*}
for a random $x, z \in \Z_p$ and $y \in \OO$. The probability is given by $p^{-4e_j}$. 
\end{itemize}
By the above computations, the probability that $X \in \HH_n(\OO)$ is given by
\begin{equation} \label{eq2e}
\prod_{i<i'}p^{-2(d_i+d_{i'})} \prod_{i} p^{-2d_i} \prod_{i, j} p^{-2(2d_i+e_j)} \prod_{j<j'} p^{-4(e_j+e_{j'})} \prod_{j} p^{-4e_j} = \left| \Gamma \right|^{-n}.
\end{equation}
Given the condition $X \in \HH_n(\OO)$, we may assume that $\overline{M}$ is zero outside of its upper left $(n-r) \times (n-r)$ minor by permuting the rows and columns of $M$. In this case, $\rank_{\F_p}(\overline{A})= n-r$ if and only if the upper left $(n-r) \times (n-r)$ minor of $\overline{A}$, which is random in $\Sym_{n-r}(\F_p)$, has a rank $n-r$. Therefore the probability that $\left< \; , \; \right>_A = \left< \; , \; \right>_M$ is given by
\begin{equation*}
\frac{\left| \GL_{n-r}(\F_p) \cap \Sym_{n-r}(\F_p) \right|}{\left| \Sym_{n-r}(\F_p) \right|} \cdot \left| \Gamma \right|^{-n}
= \prod_{i=1}^{\left \lceil \frac{n-r}{2} \right \rceil} (1 - \frac{1}{p^{2i-1}}) \left| \Gamma \right|^{-n}.
\end{equation*}
(The equality holds due to Lemma \ref{lem2e}.)

    \item Using the above results, we conclude that
    \begin{equation*}
\begin{split}
\mu_n(\Gamma, \delta) & = \prod_{i=1}^{\left \lceil \frac{n-r}{2} \right \rceil} (1 - \frac{1}{p^{2i-1}}) \left| \Gamma \right|^{-n} \cdot \frac{\left| \Gamma \right|^n}{\left| \Aut_{\OO}(\Gamma, \delta) \right|} \prod_{j=n-r+1}^{n} (1 - \frac{1}{p^j}) \\
& = \frac{1}{\left| \Aut_{\OO}(\Gamma, \delta) \right|} \prod_{j=n-r+1}^{n}(1-\frac{1}{p^j}) \prod_{i=1}^{\left \lceil \frac{n-r}{2} \right \rceil}(1 - \frac{1}{p^{2i-1}})
\end{split}    
\end{equation*}
    and the equation (\ref{eq2x1}) implies that
\begin{equation*}
\lim_{n \rightarrow \infty} \PP(\cok(A) \cong \Gamma) 
= \lim_{n \rightarrow \infty} \sum_{[\delta] \in \overline{\Phi}_{\Gamma}} \mu_n(\Gamma, \delta)
= \frac{\left| \Phi_{\Gamma} \right|}{\left| \Aut_{\OO}(\Gamma) \right|} \prod_{i=1}^{\infty}(1 - \frac{1}{p^{2i-1}}). \qedhere
\end{equation*}
\end{enumerate}
\end{proof}

For a partition $\lambda = (\lambda_1 \geq \cdots \geq \lambda_r)$ ($\lambda_r \geq 1$), an $\OO$-module of \textit{type} $\lambda$ is defined by $\prod_{i=1}^{r} \OO / \pi^{\lambda_i} \OO$. 
The moments of the cokernel of a Haar random matrix $X_n \in \HH_n(\OO)$ are given as follow. The next theorem is a special case of Theorem \ref{thm1i}.

\begin{theorem} \label{thm2g}
Let $G = G_{\lambda}$ be a finite $\OO$-module of type $\lambda = (\lambda_1 \geq \cdots \geq \lambda_r)$. 
\begin{enumerate}
    \item If $K/\Q_p$ is unramified, then
    \begin{equation} \label{eq2f}
    \lim_{n \rightarrow \infty} \EE(\# \Sur_{\OO}(\cok(X_n), G)) = p^{\sum_{i=1}^{r} (2i-1)\lambda_i}.
    \end{equation}
    
    \item If $K/\Q_p$ is ramified, then 
    \begin{equation} \label{eq2g}
    \lim_{n \rightarrow \infty} \EE(\# \Sur_{\OO}(\cok(X_n), G)) = p^{\sum_{i=1}^{r}\left ( (i-1)\lambda_i + \left \lfloor \frac{\lambda_i}{2}\right \rfloor \right )}.
    \end{equation}
\end{enumerate}
\end{theorem}

\begin{proof}
For $n \geq r$, denote $A = X_n$ for simplicity. 
Assume that $\det (A) \neq 0$ (so $\cok(A)$ is finite), which holds with probability $1$. Following the proof of \cite[Theorem 11]{CKLPW15}, the problem reduces to the computation of the probability that $A\OO^n \in D\OO^n$ where
$$
D := \diag(\pi^{e_1}, \cdots, \pi^{e_n}) \;\; (e_1 = \lambda_1, \cdots, e_r = \lambda_r, \, e_{r+1} = \cdots = e_n = 0).
$$
The condition $A\OO^n \in D\OO^n$ holds if and only if $A_{ij} \in \pi^{e_i}\OO$ for every $i \leq j$. 
\begin{itemize}
    \item For $i<j$, $A_{ij}$ is a random element of $\OO$ so 
    $$
    \PP(A_{ij} \in \pi^{e_i}\OO) = \left\{\begin{matrix}
p^{-2e_i} & (K/\Q_p \text{ is unramified}) \\ 
p^{-e_i} & (K/\Q_p \text{ is ramified})
\end{matrix}\right. .
    $$
    
    \item For $i=j$, $A_{ii}$ is a random element of $\Z_p$. If $K/\Q_p$ is unramified, then $A_{ii} \in \pi^{e_i}\OO$ if and only if $A_{ii} \in p^{e_i}\Z_p$. If $K/\Q_p$ is ramified, then $A_{ii} \in \pi^{e_i}\OO$ if and only if $A_{ii} \in p^{\left \lceil \frac{e_i}{2} \right \rceil}\Z_p$. Thus
    $$
    \PP(A_{ii} \in \pi^{e_i}\OO) = \left\{\begin{matrix}
p^{-e_i} & (K/\Q_p \text{ is unramified}) \\ 
p^{-\left \lceil \frac{e_i}{2} \right \rceil} & (K/\Q_p \text{ is ramified})
\end{matrix}\right. .
    $$
\end{itemize}
Since the probability that a random uniform element of $\Hom_{\OO}(\OO^n, G)$ is a surjection goes to $1$ as $n \rightarrow \infty$, we have
\begin{equation*}
\begin{split}
& \lim_{n \rightarrow \infty} \EE(\# \Sur_{\OO}(\cok(X_n), G)) \\
= & \lim_{n \rightarrow \infty} \left| \Sur_{\OO}(\OO^n, G) \right| \prod_{1 \leq i < j \leq n} \PP(A_{ij} \in \pi^{e_i}\OO) \prod_{1 \leq i \leq n} \PP(A_{ii} \in \pi^{e_i}\OO) \\
= & \left\{\begin{matrix}
p^{\sum_{i=1}^{r} (2i-1)\lambda_i} & (K/\Q_p \text{ is unramified}) \\
p^{\sum_{i=1}^{r}\left ( (i-1)\lambda_i + \left \lfloor \frac{\lambda_i}{2}\right \rfloor \right )} & (K/\Q_p \text{ is ramified})
\end{matrix}\right. . \qedhere
\end{split}
\end{equation*}
\end{proof}

Let $\lambda$ be a partition and $\lambda'$ be its transpose partition. Then we have $\sum_{i} (2i-1) \lambda_i = \sum_{j} \lambda_j'^2$. (This follows from the formula $\sum_{i} (i-1) \lambda_i = \sum_{j} \frac{\lambda_j'^2-\lambda_j'}{2}$ which appears in \cite[p. 717]{CKLPW15}.)
By Theorem \ref{thm2g}, the $n \rightarrow \infty$ limit of the moment $\EE(\# \Sur_{\OO}(\cok(X_n), G))$ for a finite $\OO$-module $G$ of type $\lambda$ is bounded above by $\left| \kappa \right|^{\sum_{j=1}^{\lambda_1}\frac{\lambda_j'^2}{2}}$ in both unramified and ramified cases. In the next section, we will prove that these moments determine the distribution.

\section{Moments} \label{Sec3}

Unlike the equidistributed case, it seems almost impossible to compute the distribution of the cokernel of an arbitrary $\varepsilon$-balanced matrix directly. 
The use of moments enables us to compute the distributions of random groups (in our case, random $\OO$-modules) without direct computation.
Wood \cite{Woo17, Woo19} and Nguyen-Wood \cite{NW22} proved Theorem \ref{thm1b} by computing the moments of the cokernels of $\varepsilon$-balanced matrices and showing that these moments determine the distribution.
In the remaining part of the paper, we follow the same strategy. 

In this section, we prove that the distribution of the cokernel of a Haar random Hermitian matrix over $\OO$ is uniquely determined by its moments that appear in Theorem \ref{thm2g}. Since the ring $\OO$ is a PID, the structure theorem for finitely generated modules over a PID implies that the partitions classify finite $\OO$-modules. 
Our arguments are largely based on \cite[Section 7--8]{Woo17}, so some details will be omitted.

For a partition $\lambda$, denote the finite $\OO$-module of type $\lambda$ by $G_{\lambda}$ and let $m(G_{\lambda}) := \left| \kappa \right|^{\sum_{j=1}^{\lambda_1} \frac{\lambda_j'^2}{2}}$. For partitions $\mu \leq \lambda$, let $G_{\mu, \lambda}$ be the set of $\OO$-submodules of $G_{\lambda}$ of type $\mu$. The following lemma is an analogue of \cite[Lemma 7.5]{Woo17}. 

\begin{lemma} \label{lem3a}
We have $\sum_{G \leq G_{\lambda}} m(G) \leq F^{\lambda_1} m(G_{\lambda})$ for $F := \prod_{i=1}^{\infty} (1 - 2^{-i})^{-1} \cdot \sum_{d=0}^{\infty} 2^{-\frac{d^2}{2}} > 0$.
\end{lemma}

\begin{proof}
Let $C = \prod_{i=1}^{\infty} (1 - 2^{-i})$ and 
$D = \sum_{d=0}^{\infty} 2^{-\frac{d^2}{2}}$, so that $F = C^{-1}D$. Then we have
\begin{equation} \label{eq3a}
\begin{split}
\left| G_{\mu, \lambda} \right|
& = \prod_{j \geq 1} \left ( \left| \kappa \right|^{\mu_j'\lambda_j' - \mu_j'^2} \prod_{k=1}^{\mu_j'-\mu_{j+1}'} \frac{1-\left| \kappa \right|^{-\lambda_j'+\mu_j'-k}}{1-\left| \kappa \right|^{-k}}\right ) \\
& \leq \left| \kappa \right|^{\sum_{j=1}^{\lambda_1} (\mu_j'\lambda_j' - \mu_j'^2)} \prod_{j=1}^{\lambda_1} \left ( \prod_{k=1}^{\infty} \frac{1}{1-\left| \kappa \right|^{-k}} \right ) \\
& \leq \frac{1}{C^{\lambda_1}} \left| \kappa \right|^{\sum_{j=1}^{\lambda_1} (\mu_j'\lambda_j' - \mu_j'^2)}
\end{split}    
\end{equation}
by \cite[Theorem 2.4]{HL00}. Now the equation (\ref{eq3a}) implies that
\begin{equation*}
\begin{split}
\sum_{G \leq G_{\lambda}} m(G) 
& = \sum_{\mu \leq \lambda} \left| G_{\mu, \lambda} \right| m(G_{\mu}) \\
& \leq \frac{1}{C^{\lambda_1}} \sum_{\mu \leq \lambda} \left| \kappa \right|^{\sum_{j=1}^{\lambda_1} (\mu_j'\lambda_j' - \mu_j'^2 + \frac{\mu_j'^2}{2})} \\
& = \frac{\left| \kappa \right|^{\sum_{j=1}^{\lambda_1} \frac{\lambda_j'^2}{2}}}{C^{\lambda_1}} \sum_{\mu \leq \lambda} \left| \kappa \right|^{\sum_{j=1}^{\lambda_1} - \frac{(\lambda_j' - \mu_j')^2}{2}} \\
& \leq \frac{\left| \kappa \right|^{\sum_{j=1}^{\lambda_1} \frac{\lambda_j'^2}{2}}}{C^{\lambda_1}} \sum_{d_1, \cdots, d_{\lambda_1} \geq 0} \left| \kappa \right|^{-\sum_{j=1}^{\lambda_1} \frac{d_j^2}{2}} \\
& \leq F^{\lambda_1} m(G_{\lambda}). \qedhere
\end{split}    
\end{equation*}
\end{proof}

\begin{lemma} \label{lem3b}
Let $m$ be a positive integer, $M$ be the set of partitions with at most $m$ parts and $t>1$ and $b<1$ be real numbers. For each $\mu \in M$, let $x_{\mu}$ and $y_{\mu}$ be non-negative real numbers. Suppose that for all $\lambda \in M$, 
$$
\sum_{\mu \in M} x_{\mu} t^{\sum_{i} \lambda_i \mu_i}
= \sum_{\mu \in M} y_{\mu} t^{\sum_{i} \lambda_i \mu_i}
= O(F^m t^{\sum_{i} \frac{\lambda_i^2 + b \lambda_i}{2}})
$$
for an absolute constant $F>0$. Then we have $x_{\mu} = y_{\mu}$ for all $\mu \in M$.
\end{lemma}

\begin{proof}
See \cite[Theorem 8.2]{Woo17}. The upper bound given there is of the form $O(F^m t^{\sum_{i} \frac{\lambda_i^2 - \lambda_i}{2}})$, but the same proof works for the bound $O(F^m t^{\sum_{i} \frac{\lambda_i^2 + b \lambda_i}{2}})$ for every $b<1$. 
\end{proof}

The following theorem can be proved exactly as in \cite[Theorem 8.3]{Woo17} using Lemma \ref{lem3a} and \ref{lem3b} (for $t = \left| \kappa \right|$ and $b=0$). Since the limits of the moments appear in Theorem \ref{thm2g} are bounded above by $m(G_{\lambda})$, this theorem implies that the limiting distribution of the cokernels of Haar random matrices in $\HH_n(\OO)$ is uniquely determined by their moments. 

\begin{theorem} \label{thm3c}
Let $(A_n)_{n \geq 1}$ and $(B_n)_{n \geq 1}$ be sequences of random finitely generated $\OO$-modules. Let $a$ be a non-negative integer and $\mathcal{M}_a$ be the set of finite $\OO$-modules $M$ such that $\pi^a M = 0$. Suppose that for every $G \in \mathcal{M}_a$, the limit 
$\displaystyle \lim_{n \rightarrow \infty} \PP(A_n \otimes \OO/\pi^a \OO \cong G)$ exists and we have
$$
\lim_{n \rightarrow \infty} \EE(\# \Sur_{\OO}(A_n , G)) 
= \lim_{n \rightarrow \infty} \EE(\# \Sur_{\OO}(B_n , G)) 
= O(m(G)).
$$
Then for every $G \in \mathcal{M}_a$, we have $\displaystyle \lim_{n \rightarrow \infty} \PP(A_n \otimes \OO/\pi^a \OO \cong G) = \lim_{n \rightarrow \infty} \PP(B_n \otimes \OO/\pi^a \OO \cong G)$.
\end{theorem}

\section{Universality of the cokernel: the unramified case} \label{Sec4}

In this section, we assume that $K/\Q_p$ is unramified. 
Let $m$ be a positive integer and $G$ be a finite $\OO$-module of type $\lambda = (\lambda_1 \geq \cdots \geq \lambda_r)$ such that $p^m G = 0$ (equivalently, $m \geq \lambda_1$). Let $R$ and $R_1$ be the rings $\OO / p^m \OO$ and $\Z / p^m \Z$, respectively.
Denote the image of $w$ in $R$ also by $w$. (Recall that $w$ is a fixed primitive $(p^2-1)$-th root of unity in $\overline{\Q_p}$.) 
For $X \in \HH_n(\OO)$ (resp. $X \in \HH_n(R)$), denote $X_{ij} = Y_{ij} + w Z_{ij}$ for $Y_{ij}, Z_{ij} \in \Z_p$ (resp. $Y_{ij}, \, Z_{ij} \in R_1$). Then $Z_{ii}=0$ and $X$ is determined by $n^2$ elements $Y_{ij}$ ($i \leq j$), $Z_{ij}$ ($i < j$).

\begin{definition} \label{def4a}
Let $0 < \varepsilon < 1$. A random variable $x$ in $\Z_p$ (or $R_1$) is \textit{$\varepsilon$-balanced} if $\PP (x \equiv r \,\, (\text{mod } p)) \leq 1 - \varepsilon$ for every $r \in \Z/p\Z$. 
A random matrix $X$ in $\HH_n(\OO)$ (or $\HH_n(R)$) is \textit{$\varepsilon$-balanced} if the $n^2$ elements $Y_{ij}$ ($i \leq j$), $Z_{ij}$ ($i < j$) are independent and $\varepsilon$-balanced.
\end{definition}

Let $V=R^n$ (resp. $W=R^n$) with a standard basis $v_1, \cdots, v_n$ (resp. $w_1, \cdots, w_n$) and its dual basis $v_1^*, \cdots, v_n^*$ (resp. $w_1^*, \cdots, w_n^*$). For an $\varepsilon$-balanced matrix $X_0 \in \HH_n(\OO)$, its modulo $p^m$ reduction $X \in \HH_n(R)$ is also $\varepsilon$-balanced and we have
\begin{equation} \label{eq4a}
\EE(\# \Sur_{\OO}(\cok(X_0), G))
= \EE(\# \Sur_{R}(\cok(X), G))
= \sum_{F \in \Sur_R(V, G)} \PP (FX = 0).
\end{equation}
(Here we understand $X \in \HH_n(R)$ as an element of $\Hom_R(W, V)$, so $FX \in \Hom_R(W, G)$.) Therefore, to compute the moments of $\cok(X_0)$, it is enough to compute the probability $\PP (FX = 0)$ for each $F \in \Sur_R(V, G)$.

\begin{lemma} \label{lem4b}
Let $\zeta := \zeta_{p^m} \in \C$ be a primitive $p^m$-th root of unity and $\tr : R \rightarrow R_1$ be the trace map given by $\tr (x) := x + \sigma(x)$. Then we have
\begin{equation} \label{eq4new1}
1_{FX=0} = \frac{1}{\left| G \right|^n}\sum_{C \in \Hom_R(\Hom_R(W, G), R)} \zeta^{\tr(C(FX))}. 
\end{equation}
\end{lemma}

\begin{proof}
Write $J := \Hom_R(\Hom_R(W, G), R)$ for simplicity. When $FX = 0$, the right-hand side of the equation (\ref{eq4new1}) is given by
$$
\frac{1}{\left| G \right|^n}\sum_{C \in J} \zeta^{\tr(C(0))}
= \frac{\left| J \right|}{\left| G \right|^n}
= 1.
$$
Therefore it is enough to show that if $\alpha \in \Hom_R(W, G)$ is non-zero, then 
$$
\frac{1}{\left| G \right|^n}\sum_{C \in J} \zeta^{\tr (C \alpha)} = 0.
$$
Let $t>0$ be an integer such that $p^{t-1} \alpha \neq 0$ and $p^t \alpha = 0$ in $\Hom_R(W, G)$. Without loss of generality, we may assume that $p^{t-1} \alpha(w_1) \neq 0$ (and $p^t \alpha(w_1)=0$). For each $x \in R$ with $p^t x = 0$, there exists $C_1 \in \Hom_R(G, R)$ such that $C_1(\alpha(w_1)) = x$. Since we have $G \cong \prod_{j=1}^{r} R/p^{\lambda_j}R$, $\alpha(w_1) \in G$ corresponds to $(\alpha(w_1)_1, \cdots, \alpha(w_1)_r) \in \prod_{j=1}^{r} R/p^{\lambda_j}R$ and $p^{t-1}\alpha(w_1)_k \neq 0$ for some $k$. Then $\alpha(w_1)_k = p^{\lambda_k-t}u$ for some $u \in R^{\times}$ so the map $C_1 \in \Hom_R(G, R)$ defined by $\displaystyle C_1(y_1, \cdots, y_r)=\frac{x}{p^{\lambda_k-t}u}y_k$ is well-defined and satisfies $C_1(\alpha(w_1)) = x$.

Using an isomorphism $\Hom_R(W, G) \cong G^n$ ($f \mapsto (f(w_1), \cdots, f(w_n))$, let $C \in J$ be the element which corresponds to $(C_1, 0, \cdots, 0) \in \Hom_R(G, R)^n$. Then $C \alpha = C_1(\alpha(w_1)) = x$. Conversely, if there exists $C \in J$ such that $C \alpha = x$, then $p^t x = C(p^t \alpha) = 0$. Therefore $J_x := \left\{ C \in J : C \alpha = x \right\}$ is nonempty if and only if $x \in R[p^t] := \left\{ x \in R : p^tx = 0 \right\}$. For any $x \in R[p^t]$ and $C_x \in J_x$, we have $J_x = \left\{ C_x + C_0 : C_0 \in J_0 \right\}$ so each of the sets $J_x$ has the same cardinality $\displaystyle \frac{\left| G \right|^n}{\left| R[p^t] \right|} = \frac{\left| G \right|^n}{p^{2t}}$. Now we have
\begin{equation} \label{eq4new2}
\frac{1}{\left| G \right|^n}\sum_{C \in J} \zeta^{\tr(C \alpha)}
= \frac{1}{p^{2t}} \sum_{x \in R[p^t]} \zeta^{\tr(x)}
= \frac{1}{p^{2t}} \sum_{y, z \in R_1[p^t]} \zeta^{2y+(w+w^p)z}.
\end{equation} 
(Recall that we have $\sigma(w) = w^p$ in the unramified case.) When $p$ is odd, we have
$$
\sum_{y \in R_1[p^t]}  \zeta^{2y+(w+w^p)z} = \zeta^{(w+w^p)z} \sum_{y \in R_1[p^t]}  \zeta^{2y} = 0
$$
for each $z \in R_1[p^t]$. When $p=2$, we have $w + w^2 = -1$ so
$$
\sum_{z \in R_1[p^t]}  \zeta^{2y+(w+w^p)z} = \zeta^{2y} \sum_{z \in R_1[p^t]}  \zeta^{-z} = 0
$$
for each $y \in R_1[p^t]$. Thus the right-hand side of the equation (\ref{eq4new2}) is always zero.
\end{proof}

By the above lemma, we have
\begin{equation} \label{eq4b}
\PP (FX = 0)
= \EE (1_{FX=0})
= \frac{1}{\left| G \right|^n}\sum_{C \in \Hom_R(\Hom_R(W, G), R)} \EE (\zeta^{\tr(C(FX))}).    
\end{equation}

For a finite $R$-module $G$, its conjugate ${}^{\sigma}G$ is an $R$-module which is same as $G$ as abelian groups, but the scalar multiplication is given by $r \cdot x := \sigma(r)x$. 
For every $f \in G^* := \Hom_{R}(G, R)$, there is a conjugate ${}^{\sigma}f \in ({}^{\sigma}G)^*$ defined by $({}^{\sigma}f)(x) := \sigma(f(x))$. 
Define an $R$-module structure on $G^*$ by $(rf)(x) := r f(x)$. 

\begin{lemma} \label{lem4new1}
Let $G$ and $G'$ be finite $R$-modules.
\begin{enumerate}
    \item There is a canonical $R$-module isomorphism ${}^{\sigma} (G^*) \cong ({}^{\sigma}G)^*$.

    \item There is a canonical $R$-module isomorphism $\Hom_R(G, G') \cong \Hom_R({}^{\sigma}G, {}^{\sigma}G')$.
\end{enumerate}
\end{lemma}

\begin{proof}
\begin{enumerate}
    \item Let $\phi : {}^{\sigma} (G^*) \rightarrow ({}^{\sigma}G)^*$ be the map defined by $f \mapsto {}^{\sigma}f$. It is clear that $\phi$ is a bijection. 
For $f, g \in {}^{\sigma} (G^*)$, we have $({}^{\sigma}(f+g))(x) = \sigma((f+g)(x)) = \sigma(f(x)) + \sigma(g(x)) = ({}^{\sigma}f)(x) + ({}^{\sigma}g)(x)$ so ${}^{\sigma}(f+g) = {}^{\sigma}f + {}^{\sigma}g$. 
For $f \in {}^{\sigma} (G^*)$ and $r \in R$, we have $({}^{\sigma}(r \cdot f))(x) = ({}^{\sigma}(\sigma(r) f))(x) = \sigma((\sigma(r) f)(x)) = \sigma(\sigma(r) f(x))$ 
and $(r ({}^{\sigma}f))(x) = ({}^{\sigma}f)(r \cdot x) 
= \sigma(f(\sigma(r)x)) = \sigma(\sigma(r) f(x))$ so ${}^{\sigma}(r \cdot f) = r ({}^{\sigma}f)$. Therefore $\phi$ is an $R$-module isomorphism.

    \item Let $\phi : \Hom_R(G, G') \rightarrow \Hom_R({}^{\sigma}G, {}^{\sigma}G')$ be the map defined by $\phi(f)(x)=f(x)$. It is clear that $\phi$ is an isomorphism of abelian groups. For $f \in \Hom_R(G, G')$ and $r \in R$, we have $\phi(rf)(x) = f(r \cdot x) = \phi(f)(r \cdot x) = (r\phi(f))(x)$ for every $x \in {}^{\sigma}G$ so $\phi(rf) = r\phi(f)$. \qedhere
\end{enumerate}
\end{proof}

Now we identify $V = (^{\sigma}W)^*$ (so $W \cong {}^{\sigma} (V^*) \cong ({}^{\sigma}V)^*$), $v_i = {}^{\sigma}(w_i^*)$ and $w_i = {}^{\sigma}(v_i^*)$. For $F \in \Hom_R(V, G)$ and 
$$
C \in \Hom_R(\Hom_R(W, G), R) 
\cong \Hom_R(W^*, G^*) 
\cong \Hom_R({}^{\sigma} (W^*), {}^{\sigma} (G^*))
\cong \Hom_R(V, ({}^{\sigma}G)^*),
$$
denote $e_{ij} := C(v_j)(F(v_i)) \in R$. (By abuse of notation, we denote the image of $C$ in $\Hom_R(V, ({}^{\sigma}G)^*)$ also by $C$.) 
Since $X$ is Hermitian, $X_{ji} = \sigma(X_{ij}) = Y_{ij} + w^p Z_{ij}$. Therefore
\begin{equation} \label{eq4c}
\begin{split}
\tr (C(FX)) 
& = \tr (\sum_{i, j} e_{ij}X_{ij} ) \\
& = \sum_{i=1}^{n} \tr(e_{ii}X_{ii}) + \sum_{i < j} \tr (e_{ij}(Y_{ij}+ w Z_{ij}) + e_{ji}(Y_{ij}+ w^p Z_{ij})) \\
& = \sum_{i=1}^{n} \tr(e_{ii}) Y_{ii} 
+ \sum_{i < j} \tr(e_{ij}+e_{ji})Y_{ij} 
+ \sum_{i < j} \tr(e_{ij} w +e_{ji} w^p)Z_{ij}.
\end{split}
\end{equation}
By the equations (\ref{eq4b}) and (\ref{eq4c}), we have
\begin{equation} \label{eq4d}
\begin{split}
& \PP (FX=0) \\
= \, &  \frac{1}{\left| G \right|^n} \sum_{C} 
\left ( \prod_{i} \EE(\zeta^{\tr(e_{ii}) Y_{ii}}) \right )
\left ( \prod_{i < j} \EE(\zeta^{\tr(e_{ij}+e_{ji})Y_{ij}}) \right ) \left ( \prod_{i < j} \EE(\zeta^{\tr(e_{ij}w +e_{ji} w^p)Z_{ij}}) \right ) \\
= & \frac{1}{\left| G \right|^n} \sum_{C} p_F(C)
\end{split}
\end{equation}
for
$$
p_F(C) := \left ( \prod_{i} \EE(\zeta^{\tr(e_{ii}) Y_{ii}}) \right )
\left ( \prod_{i < j} \EE(\zeta^{\tr(e_{ij}+e_{ji})Y_{ij}}) \right ) \left ( \prod_{i < j} \EE(\zeta^{\tr(e_{ij}w +e_{ji} w^p)Z_{ij}}) \right ).
$$
For every $i \leq j$, define $E(C,F,i,j) := e_{ij} + \sigma(e_{ji})$.

\begin{remark} \label{rmk4c}
Let $e_{ij} = a+w^p b$ and $e_{ji} = c+wd$ for $a,b,c,d \in R_1$. Then we have
\begin{equation*}
\begin{split}
\tr(e_{ij}+e_{ji}) & = 2(a+c)+(w+w^p)(b+d), \\
\tr(e_{ij} w +e_{ji} w^p) & = (w+w^p)(a+c)+2w^{p+1}(b+d).
\end{split}    
\end{equation*}
One can check that both are zero if and only if $a+c = b+d = 0$. (Assume that $\tr(e_{ij}+e_{ji}) = \tr(e_{ij} w +e_{ji} w^p) = 0$. If one of $a+c$ and $b+d$ is zero, then the other one should also be zero. If both $a+c$ and $b+d$ are non-zero, then we have $2(a+c) \cdot 2w^{p+1}(b+d) = (w+w^p)(b+d) \cdot (w+w^p)(a+c)$ so $(w+w^p)^2 - 2 \cdot 2w^{p+1} = (w-w^p)^2=0$, which is not true.) Since 
$$
E(C, F, i, j) = (a+c) + w^p (b+d), 
$$
this is equivalent to the condition $E(C, F, i, j)=0$.  
By \cite[Lemma 4.2]{Woo17}, we have $\displaystyle \left| p_F(C) \right| \leq \exp(-\frac{\varepsilon N}{p^{2m}})$ where $N$ is the number of the non-zero coefficients $E(C, F, i, j)$. 
\end{remark}

For $F \in \Hom_R(V, G)$ and $C \in \Hom_R(V, ({}^{\sigma}G)^*)$, define the maps $\phi_{F, C} \in \Hom_R(V, G \oplus ({}^{\sigma}G)^*)$ and $\phi_{C, F} \in \Hom_R(V, ({}^{\sigma}G)^* \oplus G)$ by $\phi_{F,C}(v) = (F(v), C(v))$ and $\phi_{C, F}(v) = (C(v), F(v))$. Then, for a map
\begin{equation*}
\begin{split}
t : (({}^{\sigma}G)^* \oplus G) \times (G \oplus ({}^{\sigma}G)^*) & \rightarrow R  \\
((\phi_1, g_1), (g_2, \phi_2)) & \mapsto \phi_1(g_2) + \sigma(\phi_2(g_1))),
\end{split}
\end{equation*}
we have $t(\phi_{C, F}(v_j), \phi_{F, C}(v_i)) = E(C, F, i, j)$. For $\nu \subset [n]$, let $V_{\nu}$ (resp. $V_{\setminus \nu}$) be an $R$-submodule of $V$ generated by $v_i$ with $i \in \nu$ (resp. $i \in [n] \setminus \nu$). The following definitions are from \cite[p.928--929]{Woo17}. 

\begin{definition} \label{def4d}
Let $0 < \gamma < 1$ be a real number. For a given $F$, we say $C$ is $\gamma$-\textit{robust} for $F$ if for every $\nu \subset [n]$ with $\left| \nu \right| < \gamma n$, we have $\ker(\phi_{C, F} \mid_{V \setminus \nu}) \neq \ker(F \mid_{V \setminus \nu})$. Otherwise, we say $C$ is $\gamma$-\textit{weak} for $F$. 
\end{definition}

\begin{definition} \label{def4e}
Let $d_0 > 0$ be a real number. An element $F \in \Hom_R(V, G)$ is called a \textit{code} of distance $d_0$ if for every $\nu \subset [n]$ with $\left| \nu \right| < d_0$, we have $F V_{\setminus \nu} = G$. 
\end{definition}

The following lemmas are analogues of \cite[Lemma 3.1 and 3.5]{Woo17}, whose proofs are also identical. Since the classification of finitely generated modules over $\OO/p^m \OO$ and finitely generated modules over $\Z/p^m\Z$ are the same, we can imitate the proof given in \cite{Woo17}. The only difference is that the equation $t(\phi_{C, F}(v_i), \phi_{F, C}(v_i)) = 2E(C, F, i, i)$ in \cite{Woo17} has changed to $t(\phi_{C, F}(v_i), \phi_{F, C}(v_i)) = E(C, F, i, i)$ in our case, which does not affect the proof at all.

\begin{lemma} \label{lem4f1}
There is a constant $C_G > 0$ such that for every $n$ and $F \in \Hom_R(V, G)$, the number of $\gamma$-weak $C$ is at most
$$
C_G \binom{n}{\left \lceil \gamma n \right \rceil - 1} \left| G \right|^{\gamma n}.
$$
\end{lemma}

\begin{lemma} \label{lem4f2}
If $F \in \Hom_R(V, G)$ is a code of distance $\delta n$ and $C \in \Hom_R(V, ({}^{\sigma}G)^*)$ is $\gamma$-robust for $F$, then
$$
\# \left\{ (i, j) : i \leq j \text{ and } E(C, F, i, j) \neq 0 \right\} \geq \frac{\gamma \delta n^2}{2 \left| G \right|^2}.
$$
\end{lemma}

Let $\HH(V)$ be the set of Hermitian pairings on $V=R^n$, i.e. the set of the maps $h : V \times V \rightarrow R$ which are bi-additive, $h(y,x) = \sigma(h(x,y))$ and $h(rx,sy)=r \sigma(s) h(x,y)$ for every $r,s \in R$ and $x, y \in V$. 
Define a map
$$
m_F : \Hom_R(V, ({}^{\sigma}G)^*) \rightarrow \HH(V)
$$
by
$$
m_F(C)(x,y) := C(x)(F(y)) + \sigma(C(y)(F(x))).
$$
It is clear that $m_F(C)$ is bi-additive and $m_F(C)(y,x) = \sigma(m_F(C)(x,y))$. For every $r, s \in R$, we have
\begin{equation*}
\begin{split}
m_F(C)(rx, sy) & = C(rx)(F(sy)) + \sigma(C(sy)(F(rx))) \\
& = rC(x)(\sigma(s) \cdot F(y)) + \sigma(sC(y)(\sigma(r) \cdot F(x))) \\
& = r \sigma(s) C(x)(F(y)) + \sigma(s \sigma(r) C(y)(F(x))) \\
& = r \sigma(s) m_F(C)(x,y)
\end{split}    
\end{equation*}
so $m_F(C)$ is an element of $\HH(V)$.

Let $e_1, \cdots, e_r$ be the canonical generators of an $R$-module $G \cong \prod_{i=1}^{r} R/p^{\lambda_i}R$ and $e_1^*, \cdots, e_r^*$ be generators of $G^*$ given by $e_i^*(\sum_{j}a_je_j) = p^{m - \lambda_i}a_i$ for $a_1, \cdots, a_r \in R$. For an element $F \in \Hom_R(V, G)$, a map $e_i^*(F) \in V^*$ is defined by $v \mapsto e_i^*(Fv)$. Now consider the following elements in $\Hom_R(V, ({}^{\sigma}G)^*)$ for a given $F$. For $\theta := w - \sigma(w)$, we have $\theta \in R^{\times}$ (see Remark \ref{rmk1g}) and $\sigma(\theta) = - \theta$.
\begin{itemize}
    \item For every $i<j$ and $c \in R/p^{\lambda_j}R$, $\displaystyle \alpha_{ij}^{c} := \frac{ce_i^*(F) ({}^{\sigma} e_j^*) - \sigma(c)e_j^*(F) ({}^{\sigma} e_i^*)}{p^{m - \lambda_i}} \in \Hom_R(V, ({}^{\sigma}G)^*)$.
    
    \item For every $i$ and $d \in R_1/p^{\lambda_i}R_1$, $\displaystyle \alpha_i^{d} := \frac{d \theta e_i^*(F) ({}^{\sigma} e_i^*)}{p^{m - \lambda_i}} \in \Hom_R(V, ({}^{\sigma}G)^*)$.
\end{itemize}

The basic properties of the elements $\alpha_{ij}^{c}$ and $\alpha_i^{d}$ are provided in the following lemmas. 

\begin{lemma} \label{lem4g}
The elements $\alpha_{ij}^{c}$ and $\alpha_i^{d}$ are contained in $\Hom_R(V, ({}^{\sigma}G)^*)$.
\end{lemma}

\begin{proof}
Since $\displaystyle \frac{e_i^*(Fv)}{p^{m - \lambda_i}} \in R$ for every $v \in V$, $\alpha_{i}^{d} \in \Hom_{\Z}(V, ({}^{\sigma}G)^*)$ for $d \in R_1$. Also, ${}^{\sigma}e_i^* \in ({}^{\sigma}G)^*$ is a $p^{\lambda_i}$-torsion element so $\alpha_{i}^{d} \in \Hom_{\Z}(V, ({}^{\sigma}G)^*)$ for $d \in R_1 / p^{\lambda_i} R_1$ is well-defined. Now we prove that the map $\alpha_{i}^{d}$ is $R$-linear. For every $r \in R$, $v \in V$ and $w \in {}^{\sigma}G$, 
$$
\alpha_i^{d}(rv)(w) 
= \frac{d \theta e_i^*(F(rv)) \sigma(e_i^*(w))}{p^{m - \lambda_i}}
= r \frac{d \theta e_i^*(Fv) \sigma(e_i^*(w))}{p^{m - \lambda_i}}
$$
and
$$
(r \cdot \alpha_i^{d}(v))(w) 
= (\alpha_i^{d}(v))(\sigma(r) w) 
= \frac{d \theta e_i^*(Fv) \sigma(e_i^*(\sigma(r)w))}{p^{m - \lambda_i}}
= r \frac{d \theta e_i^*(Fv) \sigma(e_i^*(w))}{p^{m - \lambda_i}}
$$
so $\alpha_i^{d}$ is $R$-linear. The $R$-linearity of $\alpha_{ij}^{c}$ can be proved by the same way.
\end{proof}

\begin{lemma} \label{lem4h}
The elements $\alpha_{ij}^{c}$ and $\alpha_i^{d}$ are contained in $\ker(m_F)$.
\end{lemma}

\begin{proof}
For every $x, y \in V$, we have
\begin{equation*}
\begin{split}
m_F(\alpha_{ij}^{c})(x, y)
& = \alpha_{ij}^{c}(x)(F(y)) + \sigma(\alpha_{ij}^{c}(y)(F(x))) \\
& = \frac{ce_i^*(Fx) \sigma(e_j^*(Fy)) - \sigma(c)e_j^*(Fx) \sigma(e_i^*(Fy))}{p^{m - \lambda_i}} \\
& + \sigma \left ( \frac{ce_i^*(Fy) \sigma(e_j^*(Fx)) - \sigma(c)e_j^*(Fy) \sigma(e_i^*(Fx))}{p^{m - \lambda_i}} \right ) \\
& = 0
\end{split}    
\end{equation*}
and
\begin{equation*}
\begin{split}
m_F(\alpha_i^{d})(x, y)
& = \alpha_i^{d}(x)(F(y)) + \sigma(\alpha_i^{d}(y)(F(x))) \\
& = \frac{d \theta e_i^*(Fx) \sigma(e_i^*(Fy))}{p^{m - \lambda_i}} + \sigma \left ( \frac{d \theta e_i^*(Fy) \sigma(e_i^*(Fx))}{p^{m - \lambda_i}} \right ) \\
& = 0,
\end{split}    
\end{equation*}
where the last equality follows from the facts that $\sigma(d) = d$ and $\sigma(\theta) = - \theta$.
\end{proof}

\begin{lemma} \label{lem4i}
If $FV = G$, then $\sum_{i < j} \alpha_{ij}^{c_{ij}} + \sum_{i} \alpha_i^{d_i} = 0$ if and only if each $c_{ij}$ and $d_i$ is zero. 
\end{lemma}

\begin{proof}
Assume that $\alpha := \sum_{i < j} \alpha_{ij}^{c_{ij}} + \sum_{i} \alpha_i^{d_i}$ is zero. Choose $x_1, \cdots, x_r \in V$ such that $Fx_i = e_i$ for each $i$. For every $1 \leq t \leq r$, we have
\begin{equation*}
\begin{split}
\alpha (x_t)
& = \sum_{i < t} \alpha_{it}^{c_{it}}(x_t) + \sum_{t < j} \alpha_{tj}^{c_{tj}}(x_t) + \alpha_t^{d_t}(x_t) \\
& = - \sum_{i < t} \sigma(c_{it})p^{\lambda_i - \lambda_t}  ({}^{\sigma}e_i^*) + d_t \theta ({}^{\sigma}e_t^*) + \sum_{t < j} c_{tj} ({}^{\sigma}e_j^*) \\
& = 0.
\end{split}    
\end{equation*}
For all $t < j$, $\alpha(x_t)(e_j) = p^{m-\lambda_j} c_{tj} = 0$ (in $R$) so $c_{tj}=0$ (in $R/p^{\lambda_j}R$). Since $\theta$ is a unit in $R$, the relation $\alpha(x_t)(e_t) = p^{m-\lambda_t} d_t \theta =0$ (in $R$) implies that $d_t=0$ (in $R_1/p^{\lambda_t}R_1$) for all $t$.
\end{proof}

\begin{definition} \label{def4j}
An element $C \in \Hom_R(V, ({}^{\sigma}G)^*)$ is called \textit{special} if $C = \sum_{i < j} \alpha_{ij}^{c_{ij}} + \sum_{i} \alpha_i^{d_i}$ for some $c_{ij} \in R/p^{\lambda_j}R$ and $d_i \in R_1/p^{\lambda_i}R_1$. 
\end{definition}

For a special $C$, we have $E(C, F, i, j) = m_F(C)(v_j, v_i) = 0$ for every $i \leq j$ by Lemma \ref{lem4h} so $p_F(C)=1$. The following corollary (of Lemma \ref{lem4i}) tells us that the number of special $C$ coincides with the moment appears in Theorem \ref{thm2g}(1).

\begin{corollary} \label{cor4new2}
If $FV=G$, then the number of special $C \in \Hom_R(V, ({}^{\sigma}G)^*)$ is
$$
\prod_{i<j} \left| R/p^{\lambda_j}R \right| \times \prod_{i} \left| R_1/p^{\lambda_i}R_1 \right| = p^{\sum_{i=1}^{r} (2i-1) \lambda_i}.
$$
\end{corollary}

The next proposition, which is an analogue of \cite[Lemma 3.7]{Woo17}, tells us that if $C$ is not special, it is not even close to $\ker(m_F)$.

\begin{proposition} \label{prop4k}
If $F \in \Hom_R(V, G)$ is a code of distance $\delta n$ and $C \in \Hom_R(V, ({}^{\sigma}G)^*)$ is not special, then
$$
\# \left\{ (i, j) : i \leq j \text{ and } E(C, F, i, j) \neq 0 \right\} \geq \frac{\delta n}{2}.
$$
\end{proposition}

\begin{proof}
Let $\eta := \left\{ (i, j) : i\leq j \text{ and } E(C, F, i, j) \neq 0 \right\}$ and $\nu$ be the set of all $i$ and $j$ that appear in an $(i, j) \in \eta$. Suppose that there exists a non-special $C$ such that $\left| \eta \right| < \frac{\delta n}{2}$ (so $\left| \nu \right| < \delta n$). Since $F$ is a code of distance $\delta n$, we can find $\tau \subset [n] \setminus \nu$ such that $\left| \tau \right| = r$ and $FV_{\tau} = G$. 

For $w_i \in V_{\tau}$ such that $Fw_i = e_i$, one can show that $w_1, \cdots, w_r$ form an $R$-basis of $V_{\tau}$ as in the proof of \cite[Lemma 3.6]{Woo17}. 
Let $\tau = \left\{ \tau_1, \cdots, \tau_r \right\}$ ($\tau_1 < \cdots < \tau_r$) and define
$$
z_i := \left\{\begin{matrix}
v_i & (i \notin \tau) \\
w_j & (i = \tau_j \in \tau)
\end{matrix}\right. .
$$
Then $z_1, \cdots, z_n$ is a basis of $V$. Denote its dual basis by $z_1^*, \cdots, z_n^*$.

Let $Z_1(V)$ and $Z_2(V)$ be (additive) subgroups of $\HH(V)$ defined by
\begin{equation*}
\begin{split}
Z_1(V) & := \left\{ f \in \HH(V) : f(x,y)=0 \text{ for all } x \in V, y \in V_{\tau} \right\}, \\
Z_2(V) & := \left\{ f \in \HH(V) : f(x,y)=0 \text{ for all } x, y \in V_{\tau} \right\} \supset Z_1(V)
\end{split}
\end{equation*}
and let $\HH_i(V) := \HH(V) / Z_i(V)$ for $i= 1, 2$. Also the map $m_F^i : \Hom_R(V, ({}^{\sigma}G)^*) \rightarrow \HH_i(V)$ is defined by the composition of $m_F$ and the projection $\HH(V) \rightarrow \HH_i(V)$. If $i \in \tau$ or $j \in \tau$, then $m_F(C)(v_j, v_i)=E(C, F, i, j)=0$. This shows that $m_F(C) \in Z_1(V)$ and $m_F^{1}(C)=0$. To prove the contradiction, it is enough to show that every $C \in \ker(m_F^1)$ is special, or equivalently the inequality
\begin{equation} \label{eq4e}
\left| \im(m_F^1) \right| 
= \frac{\left| \Hom_R(V, ({}^{\sigma}G)^*) \right|}{\left| \ker(m_F^1) \right|}
\leq \frac{\left| \Hom_R(V, ({}^{\sigma}G)^*) \right|}{\left| \left\{ \text{special } C \right\} \right|}
= p^{\sum_{i=1}^{r} (2n-2i+1) \lambda_i}
\end{equation}
is actually an equality. 

\begin{itemize}
    \item For each $1 \leq i < j \leq n$ and $c \in R$, there exists $f_{ij}^{c} = f_{ji}^{\sigma(c)} \in \HH(V)$ such that 
    $$
    f_{ij}^{c}(z_{i'}, z_{j'}) = \left\{\begin{matrix}
c & ((i', j')=(i,j)) \\
\sigma(c) & ((i', j')=(j, i)) \\
0 & (\text{otherwise})
\end{matrix}\right.
$$ 
for every $i'$ and $j'$. Similarly, for each $1 \leq i \leq n$ and $d \in R_1$, there exists $g_i^{d} \in \HH(V)$ such that 
$$
g_i^d (z_{i'}, z_{j'}) = \left\{\begin{matrix}
d & ((i', j')=(i,i)) \\
0 & (\text{otherwise})
\end{matrix}\right. .
$$

\item Since $V$ is a free $R$-module, the natural map $\Hom_R(V, R) \otimes ({}^{\sigma}G)^* \rightarrow \Hom_R(V, ({}^{\sigma}G)^*)$ is an isomorphism. Thus for $c \in R$ and $1 \leq b \leq a \leq r$, $c z_{\tau_a}^* \otimes {}^{\sigma}e_b^*$ is an element of $\Hom_R(V, ({}^{\sigma}G)^*)$ and
\begin{equation*}
\begin{split}
m_F(c z_{\tau_a}^* \otimes {}^{\sigma}e_b^* )(z_{\tau_i}, z_{\tau_j})
& = (c \delta_{ai} {}^{\sigma}e_b^*)({}^{\sigma} e_j) + \sigma((c \delta_{aj} {}^{\sigma}e_b^*)({}^{\sigma} e_i)) \\
& = p^{m - \lambda_b} (c \delta_{ai} \delta_{b j} + \sigma(c) \delta_{aj} \delta_{bi})
\end{split}    
\end{equation*}
so
$$
m_F^2(c z_{\tau_a}^* \otimes {}^{\sigma}e_b^* ) = \left\{\begin{matrix}
p^{m - \lambda_b}f_{\tau_b \tau_a}^{\sigma(c)} & (b<a) \\
p^{m - \lambda_b} g_{\tau_b}^{c + \sigma(c)} & (b=a)
\end{matrix}\right.
$$
as elements in $\HH_2(V)$. We also have that each element of $R_1$ can be expressed by $c + \sigma(c)$ for some $c \in R$. (For $c = x+wy$ ($x, y \in R_1$), we have $c + \sigma(c) = 2x+(w+w^p)y$. When $p$ is odd, each element of $R_1$ is of the form $2x$ for some $x \in R_1$. When $p=2$, we have $w+w^2=-1$ so $wy+\sigma(wy) = -y$ for every $y \in R_1$.) The image of $m_F^2$ contains every element of $\HH_2(V)$ of the form 
$$
f = \sum_{i<j} f_{\tau_i \tau_j}^{c_{ij}} + \sum_{i} g_{\tau_i}^{d_i}
$$ 
where $c_{ij} \in p^{m - \lambda_i} R$ and $d_i \in p^{m - \lambda_i} R_1$. As in the proof of Lemma \ref{lem4i}, one can deduce that 
$$
\sum_{i<j} f_{\tau_i \tau_j}^{c_{ij}} + \sum_{i} g_{\tau_i}^{d_i} = 0
$$ 
if and only if each $c_{ij}$ and $d_i$ is zero. This implies that the number of $f$ in $\HH_2(V)$ of the form $\sum_{i<j} f_{\tau_i \tau_j}^{c_{ij}} + \sum_{i} g_{\tau_i}^{d_i}$ is 
$$
\prod_{i<j} p^{2 \lambda_i} \cdot \prod_{i} p^{\lambda_i} = p^{\sum_{i=1}^{r} (2r-2i+1) \lambda_i}.
$$
Thus we have
\begin{equation} \label{eq4f}
\left| \im(m_F^2) \right| \geq p^{\sum_{i=1}^{r} (2r-2i+1) \lambda_i}.
\end{equation}

\item For $\ell \notin \tau$, $c \in R$ and $1 \leq b \leq r$,
\begin{equation*}
m_F(c z_{\ell}^* \otimes {}^{\sigma}e_b^*)(z_i, z_{\tau_j})
= cp^{m - \lambda_b} \delta_{\ell i} \delta_{bj}
\end{equation*}
so
$$
m_F^1(c z_{\ell}^* \otimes {}^{\sigma}e_b^*)
= f_{\ell \tau_b}^{p^{m - \lambda_b} c}
$$
as elements in $\HH_1(V)$. Let $S$ be the set of the elements of $\HH_1(V)$ of the form 
$$
f = \sum_{\ell \notin \tau} \sum_{b} f_{\ell \tau_b}^{c_{\ell b}}
$$
for some $c_{\ell b} \in p^{m-\lambda_b} R$. Then $\left| S \right| = p^{\sum_{i=1}^{r} 2(n-r) \lambda_i}$ (since $\sum_{\ell \notin \tau} \sum_{b} f_{\ell \tau_b}^{c_{\ell b}} = 0$ if and only if each $c_{\ell b}$ is zero) and $S$ is contained in the kernel of the surjective homomorphism $\im(m_F^1) \rightarrow \im(m_F^2)$, which implies that
\begin{equation} \label{eq4g}
\left| \im(m_F^1) \right| \geq p^{\sum_{i=1}^{r} 2(n-r) \lambda_i} \left| \im(m_F^2) \right|.
\end{equation}
\end{itemize}
By the equations (\ref{eq4f}) and (\ref{eq4g}), the inequality (\ref{eq4e}) should be an equality. 
\end{proof}

Now we compute the moments of the cokernel of an $\varepsilon$-balanced matrix $X \in \HH_n(\OO)$. Although the proof follows the strategy of the proof of \cite[Theorem 6.1]{Woo17}, we provide some details of the proof for the convenience of the readers. For an $\OO$-module $G$ of type $\lambda = (\lambda_1 \geq \cdots \geq \lambda_r)$, let $M_G$ be the number of special $C$ for a surjective $F \in \Hom_R(V, G)$, i.e. $M_G := p^{\sum_{i=1}^{r} (2i-1) \lambda_i}$.

\begin{lemma} \label{lem4x1}
For given $0 < \varepsilon < 1$, $\delta >0$ and $G$, there are $c, K_0 > 0$ such that the following holds: Let $X \in \HH_n(R)$ be an $\varepsilon$-balanced matrix, $F \in \Hom_R(V, G)$ be a code of distance $\delta n$ and $A \in \Hom_R((^{\sigma}V)^*, G)$. Then for each $n$, 
$$
\left| \PP(FX=0) - M_G \left| G \right|^{-n} \right|
\leq \displaystyle \frac{K_0 e^{-cn}}{\left| G \right|^n}
$$
and
$$
\PP(FX=A) \leq K_0 \left| G \right|^{-n}.
$$
\end{lemma}

\begin{proof}
By the equations (\ref{eq4b}) and (\ref{eq4c}) (replace $FX$ by $FX-A$), we have
\begin{equation} \label{eq4h}
\begin{split}
\PP (FX = A)
& = \frac{1}{\left| G \right|^n}\sum_{C} \EE (\zeta^{\tr(C(FX-A))}) \\
& = \frac{1}{\left| G \right|^n}\sum_{C} \EE (\zeta^{\tr(C(-A))}) \EE (\zeta^{\tr(C(FX))}) \\
& = \frac{1}{\left| G \right|^n} \sum_{C} \EE (\zeta^{\tr(C(-A))}) p_F(C).
\end{split}
\end{equation}
For $\gamma \in (0, \delta)$, we break the sum into $3$ pieces:
\begin{equation*}
\begin{split}
S_1 & := \left\{ C \in \Hom_R(V, ({}^{\sigma}G)^*) : C \text{ is special for } F \right\}, \\
S_2 & := \left\{ C \in \Hom_R(V, ({}^{\sigma}G)^*) : C \text{ is not special for } F \text{ and } \gamma \text{-weak for } F \right\}, \\
S_3 & := \left\{ C \in \Hom_R(V, ({}^{\sigma}G)^*) : C \text{ is } \gamma \text{-robust for } F \right\}.
\end{split}    
\end{equation*}

\begin{enumerate}[label=(\alph*)]
    \item $C \in S_1$: By Lemma \ref{lem4i}, $\left| S_1 \right| = M_G$. Since $p_F(C)=1$ for $C \in S_1$ by Lemma \ref{lem4h}, we have $\sum_{C \in S_1} \EE (\zeta^{C(-A)}) p_F(C) = M_G$ for $A=0$ and $\left| \sum_{C \in S_1} \EE (\zeta^{C(-A)}) p_F(C) \right| \leq M_G$ for any $A$.
    
    \item $C \in S_2$: By Lemma \ref{lem4f1}, $\left| S_2 \right| \leq C_G \binom{n}{\left \lceil \gamma n \right \rceil - 1} \left| G \right|^{\gamma n}$ for some constant $C_G > 0$. By Remark \ref{rmk4c} and Proposition \ref{prop4k}, we have $\left| p_F(C) \right| \leq \exp ( -\varepsilon \delta n / 2 p^{2m} )$ for every $C \in S_2$. 
    
    \item $C \in S_3$: By Remark \ref{rmk4c} and Lemma \ref{lem4f2}, we have
    $$
    \left| \sum_{C \in S_3} \EE (\zeta^{C(-A)}) p_F(C) \right| \leq \left| G \right|^n \exp (\varepsilon \gamma \delta n^2 / 2 p^{2m} \left| G \right|^2).
    $$
\end{enumerate}
Now the proof can be completed as in \cite[Lemma 4.1]{Woo17} by applying the above computations (for a sufficiently small $\gamma$) to the equation (\ref{eq4h}).
\end{proof}

For an integer $D = \prod_{i} p_i^{e_i}$, let $\ell (D) := \sum_{i} e_i$.

\begin{definition} \label{def4x2}
Assume that $\delta < \ell (\left| G \right|)^{-1}$. The \textit{depth} of an $F \in \Hom_R(V, G)$ is the maximal positive integer $D$ such that there is a $\sigma \subset [n]$ with $\left| \sigma \right| < \ell(D) \delta n$ such that $D = [G : FV_{\setminus \sigma}]$, or is $1$ if there is no such $D$. 
\end{definition}

The following lemmas are analogues of \cite[Lemma 5.2 and 5.4]{Woo17}.

\begin{lemma} \label{lem4x3}
There is a constant $K_0$ depending on $G$ such that for every $D>1$, the number of $F \in \Hom_R(V, G)$ of depth $D$ is at most
$$
K_0 \binom{n}{\left \lceil \ell(D)\delta n \right \rceil -1} \left| G \right|^n  D^{-n+\ell(D)\delta n}.
$$
\end{lemma}

\begin{lemma} \label{lem4x4}
Let $\varepsilon, \delta, G$ be as in Lemma \ref{lem4x1}. Then there exists $K_0>0$ such that if $F \in \Hom_R(V, G)$ has depth $D>1$ and $[G : FV] < D$, then for all $\varepsilon$-balanced matrix $X \in \HH_n(R)$, $$
\PP(FX=0) \leq K_0e^{-\varepsilon (1- \ell(D) \delta) n}(\left| G \right|/D)^{-(1-\ell(D) \delta)n}.
$$
\end{lemma}

\begin{proof}
We follow the proof of \cite[Lemma 5.4]{Woo17}. It is enough to show that $$
\PP(x_1f_1 \equiv g \text{ in } G/H) \leq 1 - \varepsilon \leq e^{-\varepsilon},
$$
where $H$ is an $\OO$-submodule of $G$ of index $D$, $f_1 \in G \setminus H$ and $x_1$ is a non-diagonal entry of $X$. Write $x_1 = x+wy$ for $\varepsilon$-balanced $x, y \in R_1$. 
Since $\text{Ann}_{G/H}(f_1) := \left\{ r \in R : rf_1 \in H \right\}$ is a proper ideal of $R$, it is of the form $\pi^k R$ for some $k \geq 1$. Thus the elements of the set
$$
\left \{ x \in R_1 : x_1f_1 \equiv g \text{ in } G/H \text{ for some } y \in R_1 \right \}
$$
are contained in a single equivalence class modulo $p$. Since $x$ is $\varepsilon$-balanced, we conclude that $\PP(x_1f_1 \equiv g \text{ in } G/H) \leq 1 - \varepsilon$.
\end{proof}

\begin{theorem} \label{thm4l}
Let $0 < \varepsilon < 1$ and $G$ be given. Then for any sufficiently small $c > 0$, there is a $K_0=K_{\varepsilon, G, c}>0$ such that for every positive integer $n$ and an $\varepsilon$-balanced matrix $X_0 \in \HH_n(\OO)$, 
$$
\left| \EE(\# \Sur_{\OO}(\cok(X_0), G)) - p^{\sum_{i=1}^{r} (2i-1) \lambda_i} \right| \leq K_0 e^{-cn}.
$$
In particular, the equation (\ref{eq2f}) holds for every sequence of $\varepsilon$-balanced matrices $(X_n)_{n \geq 1}$.
\end{theorem}

\begin{proof}
Throughout the proof, $K_0$ denotes a positive constant which may vary from line to line. Let $X \in \HH_n(R)$ be the reduction of $X_0 \in \HH_n(\OO)$. By the equation (\ref{eq4a}), we have
\begin{equation*}
\begin{split}
& \left| \EE(\# \Sur_{\OO}(\cok(X_0), G)) - M_G \right| \\
= & \left| \sum_{F \in \Sur_R(V, G)} \PP(FX=0) - \sum_{F \in \Hom_R(V, G)} M_G \left| G \right|^{-n} \right| \\
\leq & \sum_{F \in \Sur_R(V, G)} \left| \PP(FX=0) - M_G \left| G \right|^{-n} \right| + \sum_{F \in \Hom_R(V, G) \setminus \Sur_R(V, G)} M_G \left| G \right|^{-n}.
\end{split}    
\end{equation*}

\begin{enumerate}[label=(\alph*)]
\item By Lemma \ref{lem4x1}, we have
$$
\sum_{\substack{F \in \Sur_R(V, G) \\ F \text{ code of distance } \delta n}} \left| \PP(FX=0) - M_G \left| G \right|^{-n} \right| \leq K_0e^{-cn}.
$$

\item By Lemma \ref{lem4x3} and \ref{lem4x4}, for a sufficiently small $\delta$ we have
\begin{equation*}
\begin{split}
& \sum_{\substack{F \in \Sur_R(V, G) \\ F \text{ not code of distance } \delta n}} \PP(FX=0) \\
\leq & \sum_{\substack{D>1 \\ D \mid \# G}} K_0 \binom{n}{\left \lceil \ell(D)\delta n \right \rceil -1} \left| G \right|^n D^{-n+\ell(D)\delta n} e^{-\varepsilon (1- \ell(D) \delta) n}(\left| G \right|/D)^{-(1-\ell(D) \delta)n} \\
= & \sum_{\substack{D>1 \\ D \mid \# G}} K_0 \binom{n}{\left \lceil \ell(D)\delta n \right \rceil -1} \left| G \right|^{\ell(D) \delta n} e^{-\varepsilon (1- \ell(D) \delta) n} \\
\leq & \, K_0e^{-cn}.
\end{split}    
\end{equation*}

\item By Lemma \ref{lem4x3}, for a sufficiently small $\delta$ we have
\begin{equation*}
\begin{split}
& \sum_{\substack{F \in \Sur_R(V, G) \\ F \text{ not code of distance } \delta n}} M_G \left| G \right|^{-n} \\
\leq & \sum_{\substack{D>1 \\ D \mid \# G}} K_0 \binom{n}{\left \lceil \ell(D)\delta n \right \rceil -1} \left| G \right|^n \left| D \right|^{-n+\ell(D)\delta n} M_G \left| G \right|^{-n} \\
\leq & K_0 \binom{n}{\left \lceil \ell(\left| G \right|)\delta n \right \rceil -1} 2^{-n+\ell(\left| G \right|)\delta n} \\
\leq & \, K_0e^{-cn}.
\end{split}    
\end{equation*}

\item We also have
\begin{equation*}
\begin{split}
\sum_{F \in \Hom_R(V, G) \setminus \Sur_R(V, G)} M_G \left| G \right|^{-n}
& \leq \sum_{H \lneq G} \sum_{F \in \Sur_R(V, H)} K_0 \left| G \right|^{-n} \\
& \leq \sum_{H \lneq G} K_0  \left| H \right|^n \left| G \right|^{-n} \\
& \leq K_0 2^{-n}. \qedhere
\end{split}    
\end{equation*}
\end{enumerate}
\end{proof}

Combining the above theorem with Theorem \ref{thm2c} and \ref{thm3c}, we obtain the universality result for the distribution of the cokernels of random $p$-adic unramified Hermitian matrices. 

\begin{theorem} \label{thm4m}
For every sequence of $\varepsilon$-balanced matrices $(X_n)_{n \geq 1}$ ($X_n \in \HH_n(\OO)$), the limiting distribution of $\cok(X_n)$ is given by the equation (\ref{eq2c}).
\end{theorem}

\begin{proof}
We follow the proof of \cite[Corollary 9.2]{Woo17}. Choose a positive integer $a$ such that $p^{a-1} \Gamma = 0$. Then for any finitely generated $\OO$-module $H$, we have $H \otimes \OO/p^{a} \OO \cong \Gamma$ if and only if $H \cong \Gamma$. Let $A_n$ be the cokernel of a Haar random matrix in $\HH_n(\OO)$ and $B_n = \cok(X_n)$. Then Theorem \ref{thm2c}, \ref{thm3c} and \ref{thm4l} conclude the proof.
\end{proof}

\section{Universality of the cokernel: the ramified case} \label{Sec5}

In this section, we assume that $K/\Q_p$ is ramified. 
We say $K/\Q_p$ is of \textit{type II} if $p=2$ and $K = \Q_2(\sqrt{1+2u})$ for some $u \in \Z_2^{\times}$. Otherwise we say $K/\Q_p$ is of \textit{type I}. By the choice of the uniformizer $\pi$ in Section \ref{Sub12}, we have $\sigma(\pi) = - \pi$ if $K/\Q_p$ is of type I and $\sigma(\pi) = 2 - \pi$ if $K/\Q_p$ is of type II. Let $G$ be a finite $\OO$-module of type $\lambda = (\lambda_1 \geq \cdots \geq \lambda_r)$. Define the rings $R$, $R_1$, $R_2$ and the map $T \in \Hom_{\Z}(R, R_1)$ as follow. Recall that we have $\OO=\Z_p[\pi]$.

\begin{itemize}
    \item If $K/\Q_p$ is of type I, choose an integer $m>1$ such that $\pi^{2m-1} G = 0$. Define $R = \OO / \pi^{2m-1} \OO$, $R_1 = \Z_p/p^m \Z_p \cong \Z / p^{m}\Z$ and $R_2 = \Z_p/p^{m-1} \Z_p \cong \Z / p^{m-1}\Z$. Every element of $\OO$ is of the form $x + \pi y$ for $x, y \in \Z_p$, and the images of $x+\pi y$ and $x'+\pi y'$ in $R$ are the same if and only if $(x-x')+\pi(y-y') \in \pi^{2m-1} \OO$, which is equivalent to $x-x' \in p^m\Z_p$ and $y-y' \in p^{m-1}\Z_p$. Therefore every element of $R$ is uniquely expressed as $x + \pi y$ for some $x \in R_1$ and $y \in R_2$. Define the map $T \in \Hom_{\Z}(R, R_1)$ by $T(x + \pi y) = x$.
    
    \item If $K/\Q_p$ is of type II, choose an integer $m>0$ such that $\pi^{2m} G = 0$. Define $R = \OO / \pi^{2m} \OO$ and $R_1 = R_2 = \Z_p/p^m \Z_p \cong \Z / p^{m}\Z$. Every element of $R$ is uniquely expressed as $x + \pi y$ for some $x \in R_1$ and $y \in R_2$. Define the map $T \in \Hom_{\Z}(R, R_1)$ by $T(x + \pi y) = x+y$.
\end{itemize}

For both cases, we have $x + \sigma(x) = 2T(x)$ for every $x \in R$. This shows that it is natural to replace the trace map in Section \ref{Sec4} by the map $T$. 

For $X \in \HH_n(\OO)$, denote $X_{ij} = Y_{ij} + \pi Z_{ij}$ for $Y_{ij}, Z_{ij} \in \Z_p$. Similarly, for $X \in \HH_n(R)$, denote $X_{ij} = Y_{ij} + \pi Z_{ij}$ for $Y_{ij} \in R_1$ and $Z_{ij} \in R_2$. Then $Z_{ii}=0$ and $X$ is determined by $n^2$ elements $Y_{ij}$ ($i \leq j$), $Z_{ij}$ ($i<j$). 
Define $V$, $W$, $v_i$, $w_i$, $v_i^*$ and $w_i^*$ and identify them as in Section \ref{Sec4}.
For an $\varepsilon$-balanced matrix $X_0 \in \HH_n(\OO)$, its reduction $X \in \HH_n(R)$ is also $\varepsilon$-balanced and we have
\begin{equation} \label{eq5a}
\EE(\# \Sur_{\OO}(\cok(X_0), G))
= \EE(\# \Sur_{R}(\cok(X), G))
= \sum_{F \in \Sur_R(V, G)} \PP (FX = 0).
\end{equation}

\begin{lemma} \label{lem5a}
Let $\zeta := \zeta_{p^m} \in \C$ be a primitive $p^m$-th root of unity. Then we have
\begin{equation*}
1_{FX=0} = \frac{1}{\left| G \right|^n}\sum_{C \in \Hom_R(\Hom_R(W, G), R)} \zeta^{T(C(FX))}.    
\end{equation*}
\end{lemma}

\begin{proof}
Following the proof of Lemma \ref{lem4b}, it is enough to show that 
$$
\sum_{r \in R[\pi^t]} \zeta^{T(r)} = 0
$$
for $t > 0$. Let $\displaystyle s := \left \lceil \frac{t}{2} \right \rceil \geq 1$. If $K/\Q_p$ is of type I, then we have
\begin{equation*}
\sum_{r \in R[\pi^t]} \zeta^{T(r)}
= \sum_{x \in R_1[p^s], \, y \in R_2[p^{t-s}]} \zeta^{x}
= p^{t-s} \sum_{x \in R_1[p^s]} \zeta^{x}
= 0.
\end{equation*}
If $K/\Q_p$ is of type II, then we have
\begin{equation*}
\sum_{r \in R[\pi^t]} \zeta^{T(r)}
= \sum_{x \in R_1[p^s], \, y \in R_2[p^{t-s}]} \zeta^{x+y}
= \left ( \sum_{x \in R_1[p^s]} \zeta^{x}  \right ) \left ( \sum_{y \in R_2[p^{t-s}]} \zeta^{y}  \right )
= 0. \qedhere
\end{equation*}
\end{proof}

By the above lemma, we have
\begin{equation} \label{eq5b}
\PP (FX = 0)
= \EE (1_{FX=0})
= \frac{1}{\left| G \right|^n}\sum_{C \in \Hom_R(\Hom_R(W, G), R)} \EE (\zeta^{T(C(FX))}).    
\end{equation}
Define the maps $F$ and $C$ as in Section \ref{Sec4}. If $X \in \HH_n(R)$, then $X_{ji} = \sigma(X_{ij}) = Y_{ij} + \sigma(\pi) Z_{ij}$. Therefore
\begin{equation} \label{eq5c}
\begin{split}
T (C(FX)) 
& = T (\sum_{i, j} e_{ij}X_{ij} ) \\
& = \sum_{i=1}^{n} T(e_{ii}X_{ii}) + \sum_{i < j} T (e_{ij}(Y_{ij}+ \pi Z_{ij}) + e_{ji}(Y_{ij}+ \sigma(\pi) Z_{ij})) \\
& = \sum_{i=1}^{n} T(e_{ii}) Y_{ii} 
+ \sum_{i < j} T(e_{ij}+e_{ji})Y_{ij} 
+ \sum_{i < j} T(e_{ij} \pi +e_{ji} \sigma(\pi))Z_{ij}
\end{split}
\end{equation}
for $e_{ij} := C(v_j)(F(v_i)) \in R$. Note that if $K/\Q_p$ is of type I, then $T(e_{ij} \pi +e_{ji} \sigma(\pi)) \in pR_1$ so $T(e_{ij} \pi +e_{ji} \sigma(\pi))Z_{ij}$ is well-defined as an element of $R_1$ for $Z_{ij} \in R_2$. By the equations (\ref{eq5b}) and (\ref{eq5c}), we have
\begin{equation} \label{eq5d}
\begin{split}
& \PP (FX=0) \\
= \, &  \frac{1}{\left| G \right|^n} \sum_{C} 
\left ( \prod_{i} \EE(\zeta^{T(e_{ii})Y_{ii}}) \right )
\left ( \prod_{i < j} \EE(\zeta^{T(e_{ij}+e_{ji})Y_{ij}}) \right ) \left ( \prod_{i < j} \EE(\zeta^{T(e_{ij}\pi +e_{ji} \sigma(\pi))Z_{ij}}) \right ) \\
= & \frac{1}{\left| G \right|^n} \sum_{C} p_F(C).
\end{split}
\end{equation}
for
$$
p_F(C) := \left ( \prod_{i} \EE(\zeta^{T(e_{ii})Y_{ii}}) \right )
\left ( \prod_{i < j} \EE(\zeta^{T(e_{ij}+e_{ji})Y_{ij}}) \right ) \left ( \prod_{i < j} \EE(\zeta^{T(e_{ij}\pi +e_{ji} \sigma(\pi))Z_{ij}}) \right ).
$$
Define $E(C, F, i, i) := T(e_{ii})$ and $E(C,F,i,j) := e_{ij} + \sigma(e_{ji})$ for every $i<j$.

\begin{remark} \label{rmk5b}
For $i<j$, write $e_{ij} = a+\pi b$ and $e_{ji} = c+ \pi d$ for $a, c \in R_1$ and $b, d \in R_2$. 
\begin{itemize}
    \item Type I: 
\begin{equation*}
\begin{split}
E(C,F,i,j) & = (a+c) + \pi(b-d), \\
T(e_{ij}+e_{ji}) & = a+c, \\
T(e_{ij} \pi +e_{ji} \sigma(\pi)) & = \pi^2(b-d).
\end{split}    
\end{equation*}
Note that for $b-d \in R_2$, $\pi^2(b-d)$ is well-defined as an element of $R_1$.

    \item Type II: $\pi(\pi-2) = 2u \in R_1$ so $T(\pi^2) = T(\pi(\pi-2) + 2 \pi) = \pi^2 - 2 \pi + 2$. This implies that
\begin{equation*}
\begin{split}
E(C,F,i,j) & = (a+c+2d) + \pi(b-d), \\
T(e_{ij}+e_{ji}) & = (a+c+2d)+(b-d), \\
T(e_{ij} \pi +e_{ji} \sigma(\pi)) & = (a+c+2d)+(\pi^2 - 2 \pi + 2)(b-d).
\end{split}    
\end{equation*}
\end{itemize}
One can check that for both types,
$$
E(C, F, i, j)=0 \,\, \Leftrightarrow \,\, T(e_{ij}+e_{ji})=0 \text{ and } T(e_{ij} \pi +e_{ji} \sigma(\pi))=0.
$$
By \cite[Lemma 4.2]{Woo17}, we have $\displaystyle \left| p_F(C) \right| \leq \exp(-\frac{\varepsilon N}{p^{2m}})$ where $N$ is the number of the non-zero coefficients $E(C, F, i, j)$. 
\end{remark}

Define $\phi_{F, C}$, $\phi_{C, F}$ and $t$ as before. Then 
$$
t(\phi_{C, F}(v_j), \phi_{F, C}(v_i)) = E(C, F, i, j)
$$
for $i<j$ and 
$$
t(\phi_{C, F}(v_i), \phi_{F, C}(v_i)) = e_{ii} + \sigma(e_{ii}) = 2E(C, F, i, i).
$$

The following lemmas are analogues of \cite[Lemma 3.1 and 3.5]{Woo17}, whose proofs are also identical. Since the classification of finitely generated modules over $\OO/\pi^{2m-1} \OO$ (resp. $\OO/\pi^{2m} \OO$) and finitely generated modules over $\Z/p^{2m-1}\Z$ (resp. $\Z/p^{2m}\Z$) are the same, we can imitate the proof given in \cite{Woo17}. The $\gamma$-robustness, $\gamma$-weakness and the code of distance $d_0$ are defined as in Definition \ref{def4d} and \ref{def4e}. 

\begin{lemma} \label{lem5c1}
There is a constant $C_G > 0$ such that for every $n$ and $F \in \Hom_R(V, G)$, the number of $\gamma$-weak $C$ is at most
$$
C_G \binom{n}{\left \lceil \gamma n \right \rceil - 1} \left| G \right|^{\gamma n}.
$$
\end{lemma}

\begin{lemma} \label{lem5c2}
If $F \in \Hom_R(V, G)$ is a code of distance $\delta n$ and $C \in \Hom_R(V, ({}^{\sigma}G)^*)$ is $\gamma$-robust for $F$, then
$$
\# \left\{ (i, j) : i \leq j \text{ and } E(C, F, i, j) \neq 0 \right\} \geq \frac{\gamma \delta n^2}{2 \left| G \right|^2}.
$$
\end{lemma}

Recall that $\HH(V)$ denotes the set of Hermitian pairings on $V=R^n$. An element $f \in \HH(V)$ is uniquely determined by the values $f(v_j, v_i)$ for $1 \leq i \leq j \leq n$. Define a map
$$
m_F : \Hom_R(V, ({}^{\sigma}G)^*) \rightarrow \HH(V)
$$
by
$$
m_F(C)(v_j, v_i) := E(C, F, i, j)
$$
for every $i \leq j$. Let $m' :=2m-1$ if $K/\Q_p$ is of type I and $m' := 2m$ if $K/\Q_p$ is of type II. Let $e_1, \cdots, e_r$ be the canonical generators of an $R$-module $G \cong \prod_{i=1}^{r} R/\pi^{\lambda_i}R$ and $e_1^*, \cdots, e_r^*$ be generators of $G^*$ given by $e_i^*(\sum_{j}a_je_j) = \pi^{m' - \lambda_i}a_i$ for $a_1, \cdots, a_r \in R$. Consider the following elements in $\Hom_R(V, ({}^{\sigma}G)^*)$ for a given $F$. 
\begin{itemize}
    \item For every $i<j$ and $c \in R/\pi^{\lambda_j}R$, $\displaystyle \alpha_{ij}^{c} := \frac{ce_i^*(F) ({}^{\sigma}e_j^*)}{\pi^{m' - \lambda_i}} - \frac{\sigma(c)e_j^*(F) ({}^{\sigma}e_i^*)}{\sigma(\pi)^{m' - \lambda_i}}$.
    
    \item The definition of $\alpha_i^{d}$ for $d \in R_1 / p^{\left \lfloor \frac{\lambda_i}{2} \right \rfloor}R_1$ depends on the type of $K/\Q_p$. 
    
    \begin{itemize}
        \item (Type I) For every $i$ and $d \in R_1 / p^{\left \lfloor \frac{\lambda_i}{2} \right \rfloor} R_1$, $\displaystyle \alpha_i^{d} := \frac{d \pi e_i^*(F) ({}^{\sigma}e_i^*)}{p^{\left \lceil \frac{m'-\lambda_i}{2} \right \rceil}}$.
        
        \item (Type II) For every $i$ and $d \in R_1 / p^{\left \lfloor \frac{\lambda_i}{2} \right \rfloor} R_1$, $\displaystyle \alpha_i^{d} := \frac{d (1 - \pi) e_i^*(F) ({}^{\sigma}e_i^*)}{p^{\left \lfloor \frac{m'-\lambda_i}{2} \right \rfloor}}$.
    \end{itemize}
\end{itemize}

\begin{lemma} \label{lem5d1}
The elements $\alpha_{ij}^{c}$ and $\alpha_i^{d}$ are contained in $\Hom_R(V, ({}^{\sigma}G)^*)$.
\end{lemma}

\begin{proof}
We will prove the lemma for $\alpha_i^{d}$. The proof for $\alpha_{ij}^c$ can be done by the same way.
\begin{itemize}
    \item ($K/\Q_p$ is of type I) Since $\displaystyle \frac{\pi e_i^*(Fv)}{p^{\left \lceil \frac{m'-\lambda_i}{2} \right \rceil}} \in \pi^{1 + (m' - \lambda_i) - 2 \left \lceil \frac{m'-\lambda_i}{2} \right \rceil} R \subseteq R$ for every $v \in V$, we have $\alpha_{i}^{d} \in \Hom_{\Z}(V, ({}^{\sigma}G)^*)$ for $d \in R_1$. 
    Also, ${}^{\sigma}e_i^* \in ({}^{\sigma}G)^*$ is a $\pi^{\lambda_i}$-torsion element and
    $$
    1 + (m' - \lambda_i) - 2 \left \lceil \frac{m'-\lambda_i}{2} \right \rceil + 2 \left \lfloor \frac{\lambda_i}{2} \right \rfloor 
    = - \lambda_i - 2 \left \lceil \frac{-1-\lambda_i}{2} \right \rceil + 2 \left \lfloor \frac{\lambda_i}{2} \right \rfloor 
    \geq \lambda_i
    $$
    so $\alpha_{i}^{d} \in \Hom_{\Z}(V, ({}^{\sigma}G)^*)$ for $d \in R_1/p^{\left \lfloor \frac{\lambda_i}{2} \right \rfloor}R_1$ is well-defined. For every $r \in R$, $v \in V$ and $w \in {}^{\sigma}G$, we have
$$
\alpha_i^{d}(rv)(w) 
= \frac{d \pi e_i^*(F(rv)) \sigma(e_i^*(w))}{p^{\left \lceil \frac{m'-\lambda_i}{2} \right \rceil}}
= r \frac{d \pi e_i^*(Fv) \sigma(e_i^*(w))}{p^{\left \lceil \frac{m'-\lambda_i}{2} \right \rceil}}
$$
and
$$
(r \cdot \alpha_i^{d}(v))(w) 
= (\alpha_i^{d}(v))(\sigma(r) w) 
= \frac{d \pi e_i^*(Fv)) \sigma(e_i^*(\sigma(r)w))}{p^{\left \lceil \frac{m'-\lambda_i}{2} \right \rceil}}
= r \frac{d \pi e_i^*(Fv) \sigma(e_i^*(w))}{p^{\left \lceil \frac{m'-\lambda_i}{2} \right \rceil}}
$$
so $\alpha_i^{d}$ is $R$-linear. 

    \item ($K/\Q_p$ is of type II) Since $\displaystyle 
    \frac{d (1 - \pi) e_i^*(Fv)}{p^{\left \lfloor \frac{m'-\lambda_i}{2} \right \rfloor}} \in \pi^{(m' - \lambda_i) - 2 \left \lfloor \frac{m'-\lambda_i}{2} \right \rfloor} R \subseteq R$ for every $v \in V$, we have $\alpha_{i}^{d} \in \Hom_{\Z}(V, ({}^{\sigma}G)^*)$ for $d \in R_1$. 
    Also, ${}^{\sigma}e_i^* \in ({}^{\sigma}G)^*$ is a $\pi^{\lambda_i}$-torsion element and
    $$
    (m' - \lambda_i) - 2 \left \lfloor \frac{m'-\lambda_i}{2} \right \rfloor + 2 \left \lfloor \frac{\lambda_i}{2} \right \rfloor 
    = - \lambda_i - 2 \left \lfloor \frac{-\lambda_i}{2} \right \rfloor + 2 \left \lfloor \frac{\lambda_i}{2} \right \rfloor
    \geq \lambda_i
    $$
    so $\alpha_{i}^{d} \in \Hom_{\Z}(V, ({}^{\sigma}G)^*)$ for $d \in R_1 / \pi^{\lambda_i} R_1$ is well-defined. For every $r \in R$, $v \in V$ and $w \in {}^{\sigma}G$, we have
$$
\alpha_i^{d}(rv)(w) 
= \frac{d (1 - \pi) e_i^*(F(rv)) \sigma(e_i^*(w))}{p^{\left \lfloor \frac{m'-\lambda_i}{2} \right \rfloor}}
= r \frac{d (1 - \pi) e_i^*(Fv) \sigma(e_i^*(w))}{p^{\left \lfloor \frac{m'-\lambda_i}{2} \right \rfloor}}
$$
and
$$
(r \cdot \alpha_i^{d}(v))(w) 
= (\alpha_i^{d}(v))(\sigma(r) w) 
= \frac{d (1 - \pi) e_i^*(Fv) \sigma(e_i^*(\sigma(r)w))}{p^{\left \lfloor \frac{m'-\lambda_i}{2} \right \rfloor}}
= r \frac{d (1 - \pi) e_i^*(Fv) \sigma(e_i^*(w))}{p^{\left \lfloor \frac{m'-\lambda_i}{2} \right \rfloor}}
$$
so $\alpha_i^{d}$ is $R$-linear. \qedhere
\end{itemize}
\end{proof}

\begin{lemma} \label{lem5d}
The elements $\alpha_{ij}^{c}$ and $\alpha_i^{d}$ are contained in $\ker(m_F)$.
\end{lemma}

\begin{proof}
For every $1 \leq k<l \leq n$, 
\begin{equation*}
\begin{split}
m_F(\alpha_{ij}^{c})(v_l, v_k)
& = \alpha_{ij}^{c}(v_l)(Fv_k) + \sigma(\alpha_{ij}^{c}(v_k)(Fv_l)) \\
& = \frac{ce_i^*(Fv_l) \sigma(e_j^*(Fv_k))}{\pi^{m' - \lambda_i}} - \frac{\sigma(c)e_j^*(Fv_l) \sigma(e_i^*(Fv_k))}{\sigma(\pi)^{m' - \lambda_i}} \\
& + \sigma \left ( \frac{ce_i^*(Fv_k) \sigma(e_j^*(Fv_l))}{\pi^{m' - \lambda_i}} - \frac{\sigma(c)e_j^*(Fv_k) \sigma(e_i^*(Fv_l))}{\sigma(\pi)^{m' - \lambda_i}} \right ) \\
& = 0.
\end{split}    
\end{equation*}
Since $\pi + \sigma(\pi)=0$ for type I and $(1- \pi) + \sigma(1-\pi)=0$ for type II, $m_F(\alpha_{i}^{d})(v_l, v_k)=0$ for both types. 
For every $1 \leq k \leq n$,
\begin{equation*}
\begin{split}
m_F(\alpha_{ij}^{c})(v_k, v_k)
& = T \left ( \frac{ce_i^*(Fv_k) \sigma(e_j^*(Fv_k))}{\pi^{m' - \lambda_i}} - \frac{\sigma(c)e_j^*(Fv_k) \sigma(e_i^*(Fv_k))}{\sigma(\pi)^{m' - \lambda_i}} \right ) \\
& = T(u - \sigma(u)) \;\; (u := \frac{ce_i^*(Fv_k) \sigma(e_j^*(Fv_k))}{\pi^{m' - \lambda_i}}) \\
& = 0, \\
m_F(\alpha_{i}^{d})(v_k, v_k)
& = \frac{d e_i^*(Fv_k) \sigma(e_i^*(Fv_k))}{p^{\left \lceil \frac{m'-\lambda_i}{2} \right \rceil}} T(\pi) = 0 \;\; (\text{type I}), \\
m_F(\alpha_{i}^{d})(v_k, v_k)
& = \frac{d e_i^*(Fv_k) \sigma(e_i^*(Fv_k))}{p^{\left \lfloor \frac{m'-\lambda_i}{2} \right \rfloor}} T(1-\pi) = 0 \;\; (\text{type II}). \qedhere
\end{split}    
\end{equation*}
\end{proof}

\begin{lemma} \label{lem5e}
If $FV = G$, then $\sum_{i < j} \alpha_{ij}^{c_{ij}} + \sum_{i} \alpha_i^{d_i} = 0$ if and only if each $c_{ij}$ and $d_i$ is zero. 
\end{lemma}

\begin{proof}
Assume that $\alpha := \sum_{i < j} \alpha_{ij}^{c_{ij}} + \sum_{i} \alpha_i^{d_i}$ is zero. Choose $x_1, \cdots, x_r \in V$ such that $Fx_i = e_i$ for each $i$. For each $1 \leq t \leq r$, we have
\begin{equation*}
\begin{split}
\alpha (x_t)
& = \sum_{i < t} \alpha_{it}^{c_{it}}(x_t) + \sum_{t < j} \alpha_{tj}^{c_{tj}}(x_t) + \alpha_t^{d_t}(x_t) \\
& = \left\{\begin{matrix}
- \sum_{i < t} \frac{\sigma(c_{it}) \pi^{m'-\lambda_t}}{\sigma(\pi)^{m'-\lambda_i}} ({}^{\sigma}e_i^*) 
+ \frac{d_t \pi^{m'-\lambda_t+1}}{p^{\left \lceil \frac{m'-\lambda_t}{2} \right \rceil}} ({}^{\sigma}e_t^*) 
+ \sum_{t < j} c_{tj} ({}^{\sigma}e_j^*) & (\text{type I}) \\
- \sum_{i < t} \frac{\sigma(c_{it}) \pi^{m'-\lambda_t}}{\sigma(\pi)^{m'-\lambda_i}} ({}^{\sigma}e_i^*) 
+ \frac{d_t (1 - \pi)\pi^{m'-\lambda_t}}{p^{\left \lfloor \frac{m'-\lambda_t}{2} \right \rfloor}} ({}^{\sigma}e_t^*) 
+ \sum_{t < j} c_{tj} ({}^{\sigma}e_j^*) & (\text{type II}) 
\end{matrix}\right. \\
& = 0.
\end{split}    
\end{equation*}
For all $t<j$, $\alpha(x_t)(e_j) = \pi^{m'-\lambda_j} c_{tj} =0$ (in $R=\OO/\pi^{m'}\OO$) so $c_{tj}=0$ (in $R/\pi^{\lambda_j}R$). If $K/\Q_p$ is of type I, then
\begin{equation*}
\begin{split}
& \alpha(x_t)(e_t)= \pi^{m'-\lambda_t} \frac{d_t \pi^{m'-\lambda_t+1}}{p^{\left \lceil \frac{m'-\lambda_t}{2} \right \rceil}} = 0 \text{ in } R \\
\Leftrightarrow \;\; & 2v_p(d_t) + ((2m-1)-\lambda_t+1) - 2 \left \lceil \frac{2m-1-\lambda_t}{2} \right \rceil \geq \lambda_t \\
\Leftrightarrow \;\; & v_p(d_t) \geq \lambda_t + \left \lceil -\frac{\lambda_t+1}{2} \right \rceil = \left \lfloor \frac{\lambda_t}{2} \right \rfloor
\end{split}    
\end{equation*}
($v_p(d_t)$ denotes the exponent of $p$ in $d_t$) so $d_t=0$. Similarly, if $K/\Q_p$ is of type II, then
\begin{equation*}
\begin{split}
& \alpha(x_t)(e_t)= \pi^{m'-\lambda_t} \frac{d_t (1 - \pi)\pi^{m'-\lambda_t}}{p^{\left \lfloor \frac{m'-\lambda_t}{2} \right \rfloor}}= 0 \text{ in } R \\
\Leftrightarrow \;\; & 2v_p(d_t) + (2m-\lambda_t) - 2 \left \lfloor \frac{2m-\lambda_t}{2} \right \rfloor \geq \lambda_t \\
\Leftrightarrow \;\; & v_p(d_t) \geq \lambda_t + \left \lfloor -\frac{\lambda_t}{2} \right \rfloor = \left \lfloor \frac{\lambda_t}{2} \right \rfloor
\end{split}    
\end{equation*}
so $d_t=0$.
\end{proof}

\begin{definition} \label{def5f}
An element $C \in \Hom_R(V, ({}^{\sigma}G)^*)$ is called \textit{special} if $C$ is of the form $\sum_{i < j} \alpha_{ij}^{c_{ij}} + \sum_{i} \alpha_i^{d_i}$ for some $c_{ij} \in R/\pi^{\lambda_j}R$ and $d_i \in R_1/p^{\left \lfloor \frac{\lambda_i}{2} \right \rfloor}R_1$. 
\end{definition}

For a special $C$, we have $E(C, F, i, j) = m_F(C)(v_j, v_i) = 0$ for every $i \leq j$ by Lemma \ref{lem5d} so $p_F(C)=1$. Moreover, Lemma \ref{lem5e} implies that the number of special $C$ is 
$$
\prod_{i<j} \left| R/\pi^{\lambda_j}R \right| \times \prod_{i} \left| R_1/p^{\left \lfloor \frac{\lambda_i}{2} \right \rfloor}R_1 \right| = p^{\sum_{i=1}^{r}\left ( (i-1)\lambda_i + \left \lfloor \frac{\lambda_i}{2}\right \rfloor \right )}
$$
when $FV=G$. This number coincides with the moment appears in Theorem \ref{thm2g}(2). Now we prove that for a non-special $C$, there are linearly many non-zero coefficients $E(C, F, i, j)$.

\begin{proposition} \label{prop5g}
If $F \in \Hom_R(V, G)$ is a code of distance $\delta n$ and $C \in \Hom_R(V, ({}^{\sigma}G)^*)$ is not special, then
$$
\# \left\{ (i, j) : i \leq j \text{ and } E(C, F, i, j) \neq 0 \right\} \geq \frac{\delta n}{2}.
$$
\end{proposition}

\begin{proof}
Define $f_{ij}^{c} = f_{ji}^{\sigma(c)}$ ($c \in R$), $g_i^{d}$ ($d \in R_1$), $\HH_t(V)$ and $m_F^t : \Hom_R(V, (^{\sigma}G)^*) \rightarrow \HH_t(V)$ ($t=1, 2$) as in the proof of Proposition \ref{prop4k}. Following the argument of the of Proposition \ref{prop4k}, the proof reduces to show that the inequality
\begin{equation} \label{eq5e}
\left| \im(m_F^1) \right| 
\leq \frac{\left| \Hom_R(V, ({}^{\sigma}G)^*) \right|}{\left| \left\{ \text{special } C \right\} \right|}
= p^{\sum_{i=1}^{r} \left ( (n-i) \lambda_i + \left \lceil \frac{\lambda_i}{2} \right \rceil \right ) }
\end{equation}
is actually an equality. 

\begin{itemize}
    \item Since $V$ is a free $R$-module, the natural map $\Hom_R(V, R) \otimes ({}^{\sigma}G)^* \rightarrow \Hom_R(V, ({}^{\sigma}G)^*)$ is an isomorphism. Thus for $c \in R$ and $1 \leq b \leq a \leq r$, $c z_{\tau_a}^* \otimes {}^{\sigma}e_b^*$ is an element of $\Hom_R(V, ({}^{\sigma}G)^*)$ and
\begin{equation*}
m_F(c z_{\tau_a}^* \otimes {}^{\sigma}e_b^* )(z_{\tau_i}, z_{\tau_j})
= \left\{\begin{matrix}
c \delta_{ai} \delta_{b j} \sigma(\pi)^{m'-\lambda_b} + \sigma(c)\delta_{aj} \delta_{bi} \pi^{m'-\lambda_b} & (i<j) \\
T(c \delta_{ai} \delta_{bi} \sigma(\pi)^{m'-\lambda_b}) & (i=j)
\end{matrix}\right.
\end{equation*}
so
$$
m_F^2(c z_{\tau_a}^* \otimes {}^{\sigma}e_b^* ) = \left\{\begin{matrix}
f_{\tau_b \tau_a}^{\pi^{m'-\lambda_b} \sigma(c)} & (b<a) \\
g_{\tau_b}^{T(\pi^{m'-\lambda_b} \sigma(c))} & (b=a)
\end{matrix}\right.
$$
as elements in $\HH_2(V)$. Since we have
$$
\left\{ T(\pi^{m'-\lambda_b} \sigma(c)) : c \in R \right\} = \left\{\begin{matrix}
p^{\left \lceil \frac{m'-\lambda_b}{2} \right \rceil} R_1 = p^{m - \left \lceil \frac{\lambda_b}{2} \right \rceil}R_1 & (\text{type I}) \\
p^{\left \lfloor \frac{m'-\lambda_b}{2} \right \rfloor} R_1 = p^{m - \left \lceil \frac{\lambda_b}{2} \right \rceil}R_1 & (\text{type II})
\end{matrix}\right. ,
$$
the image of $m_F^2$ contains every element of $\HH_2(V)$ of the form 
$$
f = \sum_{i<j} f_{\tau_i \tau_j}^{c_{ij}} + \sum_{i} g_{\tau_i}^{d_i}
$$
where $c_{ij} \in \pi^{m'-\lambda_i} R$ and $d_i \in p^{m - \left \lceil \frac{\lambda_i}{2} \right \rceil} R_1$. As in the proof of Lemma \ref{lem5e}, one can deduce that 
$$
\sum_{i<j} f_{\tau_i \tau_j}^{c_{ij}} + \sum_{i} g_{\tau_i}^{d_i} = 0
$$ 
if and only if each $c_{ij}$ and $d_i$ is zero. This implies that the number of $f$ in $\HH_2(V)$ of the form $\sum_{i<j} f_{\tau_i \tau_j}^{c_{ij}} + \sum_{i} g_{\tau_i}^{d_i}$ is
$$
\prod_{i<j} p^{\lambda_i} \cdot \prod_{i} p^{\left \lceil \frac{\lambda_i}{2} \right \rceil} = p^{\sum_{i=1}^{r}\left ( (r-i)\lambda_i + \left \lceil \frac{\lambda_i}{2} \right \rceil \right ) }.
$$
Thus we have
\begin{equation} \label{eq5f}
\left| \im(m_F^2) \right| \geq p^{\sum_{i=1}^{r} \left ( (r-i)\lambda_i + \left \lceil \frac{\lambda_i}{2} \right \rceil \right )}.
\end{equation}

    \item For $\ell \notin \tau$, $c \in R$ and $1 \leq b \leq r$,
\begin{equation*}
m_F(c z_{\ell}^* \otimes {}^{\sigma}e_b^*)(z_i, z_{\tau_j})
= c \sigma(\pi^{m' - \lambda_b}) \delta_{\ell i} \delta_{bj}
\end{equation*}
so
$$
m_F^1(c z_{\ell}^* \otimes {}^{\sigma}e_b^*)
= f_{\tau_b \ell}^{\pi^{m'-\lambda_b} \sigma(c)}
$$
as elements in $\HH_1(V)$. Let $S$ be the set of the elements of $\HH_1(V)$ of the form 
$$
f = \sum_{\ell \notin \tau} \sum_{b} f_{\tau_b \ell}^{c_{b \ell}}
$$
for some $c_{b \ell} \in \pi^{m'-\lambda_b} R$. Then $\left| S \right| = p^{\sum_{i=1}^{r} (n-r) \lambda_i}$ (since $\sum_{\ell \notin \tau} \sum_{b} f_{\tau_b \ell}^{c_{b \ell}} = 0$ if and only if each $c_{b \ell}$ is zero) and $S$ is contained in the kernel of the surjective homomorphism $\im(m_F^1) \rightarrow \im(m_F^2)$, which implies that
\begin{equation} \label{eq5g}
\left| \im(m_F^1) \right| \geq p^{\sum_{i=1}^{r} (n-r) \lambda_i} \left| \im(m_F^2) \right|.
\end{equation}
\end{itemize}
By the equations (\ref{eq5f}) and (\ref{eq5g}), the inequality (\ref{eq5e}) should be an equality. 
\end{proof}

Now we compute the moments of the cokernel of an $\varepsilon$-balanced matrix $X \in \HH_n(\OO)$. As in the unramified case, we provide some details of the proof for the convenience of the readers. For an $\OO$-module $G$ of type $\lambda = (\lambda_1 \geq \cdots \geq \lambda_r)$, let $M_G$ be the number of special $C$ for a surjective $F \in \Hom_R(V, G)$, i.e.  $M_G := p^{\sum_{i=1}^{r}\left ( (i-1)\lambda_i + \left \lfloor \frac{\lambda_i}{2}\right \rfloor \right )}$. 

\begin{lemma} \label{lem5x1}
For given $0 < \varepsilon < 1$, $\delta >0$ and $G$, there are $c, K_0 > 0$ such that the following holds: Let $X \in \HH_n(R)$ be an $\varepsilon$-balanced matrix, $F \in \Hom_R(V, G)$ be a code of distance $\delta n$ and $A \in \Hom_R((^{\sigma}V)^*, G)$. Then for each $n$, 
$$
\left| \PP(FX=0) - M_G \left| G \right|^{-n} \right|
\leq \displaystyle \frac{K_0 e^{-cn}}{\left| G \right|^n}
$$
and
$$
\PP(FX=A) \leq K_0 \left| G \right|^{-n}.
$$
\end{lemma}

\begin{proof}
By the equations (\ref{eq5b}) and (\ref{eq5c}) (replace $FX$ by $FX-A$), we have
\begin{equation} \label{eq5h}
\begin{split}
\PP (FX = A)
& = \frac{1}{\left| G \right|^n}\sum_{C} \EE (\zeta^{T(C(FX-A))}) \\
& = \frac{1}{\left| G \right|^n}\sum_{C} \EE (\zeta^{T(C(-A))}) \EE (\zeta^{\tr(C(FX))}) \\
& = \frac{1}{\left| G \right|^n} \sum_{C} \EE (\zeta^{T(C(-A))}) p_F(C).
\end{split}
\end{equation}
For $\gamma \in (0, \delta)$, we break the sum into $3$ pieces:
\begin{equation*}
\begin{split}
S_1 & := \left\{ C \in \Hom_R(V, ({}^{\sigma}G)^*) : C \text{ is special for } F \right\}, \\
S_2 & := \left\{ C \in \Hom_R(V, ({}^{\sigma}G)^*) : C \text{ is not special for } F \text{ and } \gamma \text{-weak for } F \right\}, \\
S_3 & := \left\{ C \in \Hom_R(V, ({}^{\sigma}G)^*) : C \text{ is } \gamma \text{-robust for } F \right\}.
\end{split}    
\end{equation*}

\begin{enumerate}[label=(\alph*)]
    \item $C \in S_1$: By Lemma \ref{lem5e}, $\left| S_1 \right| = M_G$. Since $p_F(C)=1$ for $C \in S_1$ by Lemma \ref{lem5d}, we have $\sum_{C \in S_1} \EE (\zeta^{C(-A)}) p_F(C) = M_G$ for $A=0$ and $
    \left| \sum_{C \in S_1} \EE (\zeta^{C(-A)}) p_F(C) \right| \leq M_G$ for any $A$.
    
    \item $C \in S_2$: By Lemma \ref{lem5c1}, $\left| S_2 \right| \leq C_G \binom{n}{\left \lceil \gamma n \right \rceil - 1} \left| G \right|^{\gamma n}$ for some constant $C_G > 0$. By Remark \ref{rmk5b} and Proposition \ref{prop5g}, we have $\left| p_F(C) \right| \leq \exp ( -\varepsilon \delta n / 2 p^{2m} )$ for every $C \in S_2$. 
    
    \item $C \in S_3$: By Remark \ref{rmk5b} and Lemma \ref{lem5c2},
    $$
    \left| \sum_{C \in S_3} \EE (\zeta^{C(-A)}) p_F(C) \right| \leq \left| G \right|^n \exp (\varepsilon \gamma \delta n^2 / 2 p^{2m} \left| G \right|^2).
    $$
\end{enumerate}
Now the proof can be completed as in \cite[Lemma 4.1]{Woo17} by applying the above computations (for a sufficiently small $\gamma$) to the equation (\ref{eq5h}).
\end{proof}

Recall that for an integer $D = \prod_{i} p_i^{e_i}$, we have defined $\ell (D) := \sum_{i} e_i$ in Section \ref{Sec4}. The \textit{depth} of $F \in \Hom_R(V, G)$ is defined exactly as in Definition \ref{def4x2}. The next lemmas are analogues of \cite[Lemma 5.2 and 5.4]{Woo17}.

\begin{lemma} \label{lem5x3}
There is a constant $K_0$ depending on $G$ such that for every $D>1$, the number of $F \in \Hom_R(V, G)$ of depth $D$ is at most
$$
K_0 \binom{n}{\left \lceil \ell(D)\delta n \right \rceil -1} \left| G \right|^n  D^{-n+\ell(D)\delta n}.
$$
\end{lemma}

\begin{lemma} \label{lem5x4}
Let $\varepsilon, \delta, G$ be as in Lemma \ref{lem5x1}. Then there exists $K_0>0$ such that if $F \in \Hom_R(V, G)$ has depth $D>1$ and $[G : FV] < D$, then for all $\varepsilon$-balanced matrix $X \in \HH_n(R)$, $$
\PP(FX=0) \leq K_0e^{-\varepsilon (1- \ell(D) \delta) n}(\left| G \right|/D)^{-(1-\ell(D) \delta)n}.
$$
\end{lemma}

\begin{proof}
The proof is same as Lemma \ref{lem4x4}. In the ramified case, one can write $x_1 = x+\pi y$ for $\varepsilon$-balanced $x \in R_1$ and $y \in R_2$. For an $\OO$-submodule $H$ in $G$ of index $D$, $f_1 \in G \setminus H$ and a non-diagonal entry $x_1$ of $X$, the elements of the set
$$
\left \{ x \in R_1 : x_1f_1 \equiv g \text{ in } G/H \text{ for some } y \in R_2 \right \}
$$
are contained in a single equivalence class modulo $p$. Since $x$ is $\varepsilon$-balanced, we conclude that $\PP(x_1f_1 \equiv g \text{ in } G/H) \leq 1 - \varepsilon$.
\end{proof}

The following theorem can be proved exactly as in Theorem \ref{thm4l}. (Replace the Lemma \ref{lem4x1}, \ref{lem4x3} and \ref{lem4x4} to the Lemma \ref{lem5x1}, \ref{lem5x3} and \ref{lem5x4}, respectively.)

\begin{theorem} \label{thm5h}
Let $0 < \varepsilon < 1$ and $G$ be given. Then for any sufficiently small $c > 0$, there is a $K_0=K_{\varepsilon, G, c}>0$ such that for every positive integer $n$ and an $\varepsilon$-balanced matrix $X_0 \in \HH_n(\OO)$, 
$$
\left| \EE(\# \Sur_{\OO}(\cok(X_0), G)) - p^{\sum_{i=1}^{r}\left ( (i-1)\lambda_i + \left \lfloor \frac{\lambda_i}{2}\right \rfloor \right )} \right| \leq K_0 e^{-cn}.
$$
In particular, the equation (\ref{eq2g}) holds for every sequence of $\varepsilon$-balanced matrices $(X_n)_{n \geq 1}$.
\end{theorem}

Combining the above theorem with Theorem \ref{thm2c} and \ref{thm3c}, we obtain the universality result for the distribution of the cokernels of random $p$-adic ramified Hermitian matrices. 

\begin{theorem} \label{thm5i}
For every sequence of $\varepsilon$-balanced matrices $(X_n)_{n \geq 1}$ ($X_n \in \HH_n(\OO)$), the limiting distribution of $\cok(X_n)$ is given by the equation (\ref{eq2d}).
\end{theorem}

\begin{proof}
Choose a positive integer $a$ such that $\pi^{a-1} \Gamma = 0$. Then for any finitely generated $\OO$-module $H$, we have $H \otimes \OO/\pi^{a} \OO \cong \Gamma$ if and only if $H \cong \Gamma$. Let $A_n$ be the cokernel of a Haar random matrix in $\HH_n(\OO)$ and $B_n = \cok(X_n)$. Then Theorem \ref{thm2c}, \ref{thm3c} and \ref{thm5h} conclude the proof.
\end{proof}

\section*{Acknowledgments}

The author is supported by a KIAS Individual Grant (SP079601) via the Center for Mathematical Challenges at Korea Institute for Advanced Study. We thank Jacob Tsimerman and Myungjun Yu for their helpful comments.


\end{document}